\UseRawInputEncoding
\documentclass[12pt]{article}
\usepackage{amssymb,latexsym}

\newcommand{\yuankuo}[1]{\left( #1\right) }

\newcommand{\F}{F^{\prime}\yuankuo{\frac{|du|^2}{2}}}

\usepackage{color}

\usepackage{pdfsync}
\usepackage[colorlinks, backref=page
]{hyperref}
\usepackage{amsmath,amsthm}
\usepackage{indentfirst}
\theoremstyle{plain}
\newtheorem{thm}{Theorem}[section]
\newtheorem{cor}{Corollary}[section]
\newtheorem{lem}{Lemma}[section]

\theoremstyle{definition}
\newtheorem{defn}{Definition}[section]

\newtheorem{rem}{Remark}
\numberwithin{equation}{section}
\allowdisplaybreaks
\def \d {\mathrm{d}}

\def \vol{\mathrm{Vol}}

\title{Liouville type theorem for several  generalized  maps between Riemannian manifold }

\author{Xiangzhi Cao\thanks{School of Information Engineering, Nanjing Xiaozhuang University, Nanjing 211171, China}\thanks{This work is surported by General Project of Basic Science (Natural Science) Research of  Universities  in Jiangsu province (Grant No.
		22KJD110004)}}

\begin{document}
	\maketitle
	\tableofcontents
	\begin{abstract}		
		In this paper, we mainly derive monotonicity formula of generalized map  using conservation law, including $\phi$-$F$ harmonic map coupled with $\phi$-$F$ symphonic map with $m$ form and potential from metric measure space, $ p $ harmonic map with potential , $ V $ harmonic map with potential. As an corollary, we can derive Liouville theorem for these maps under some  finite energy conditons.  We also get  Liouville type theorem for $\phi$-$F$ harmonic map coupled with $\phi$-$F$ symphonic map under asymptotic conditon  on metric measure space.  We also get Liouville  theorem for  $\phi$-$F$-$V$-harmonic maps in terms of the upper bound of  Ricci curvature and the bound about sectional curvature on metric measure space. We also get Liouville  theorem for $\phi$-$ F $-harmonic map without using monotonicity formula on metric measure space.	
	\end{abstract}
	
	{\small
		
	}
	
	\section{Introduction }
	Liouville type theorem of harmonic map is an active field in geometric analysis. One of  the usual strategies to obtain Liouville theorem for harmonic map is to use  monotonicity formulua which can be obtainded if  one  can prove the nonexistence of stable harmonic map between manifold with certain curvature conditons on the domain manifold or the target manifold, we can also get the Liouville type theorem, such as \cite{leung1982stability} \cite{xin1980some}.
	
	Harmonic map has many generaliztions, one of them is $ F $ harmonic map.
	Ara \cite{Mitsunori1999Geometry} introduced the $F$ harmonic map and studied the stability of $F$ harmonic map. Later, Ara \cite{ara2001stability,ara2001instability} further investigated the stability and nonstability phenomena of $ F $ harmonic map.  Liu \cite{Liu2005} obtained the Liouville theorem for F harmonic map via conservation law. Dong \cite{MR3449358} also obtained the constant property of $ F $ harmonic map via conservation law. 	Dong  and Wei \cite{2011On}
	used the stress engergy of  differential forms to get the monotonicity formula, furthermore, they derived the Liouville theorem for $  F $-Yang-Mills feild.  Zhou et al. \cite{zhou} \cite{li2012monotonicity} also investigated the Liouville theorem of F-harmonic map via coservation law.	Weil \cite{wei2021dualities} derived the monotonicity formula for vector bundle valued forms under some new curvature condition about ridial curvature. 
	The reason that conservation law is useful for getting Liouville theorem  is that it can be used to establish monotonicity formula, once imposing some conditon on the energy.

	Nakauchi et al. \cite{kawai2011some,nakauchi2011,nakauchi2011variational} studied symphonic map between manifolds,  which is the critical point of  engergy  functional
	\begin{equation*}
		\begin{split}
			\int_M \frac{|u^{*}h |^2}{4}dv. 
		\end{split}
	\end{equation*}
	Motived by  Nakauchi and coauthors \cite{kawai2011some,nakauchi2011,nakauchi2011variational},	Han \cite{Han2021} studied the monotonicity formular of the funtional \cite{Antonishin1992}
	\begin{equation*}
		\begin{split}
			\int_M F(\frac{|u^{*} h|^2}{4})dv. 
		\end{split}
	\end{equation*}
	whose critical point is refered to as F stationary map in \cite{Antonishin1992}. When $ F(x)=x, $
it is also called symphonic map.	Han and Feng \cite{han2014monotonicity}  studied   monotonicity and stability of $ F $-stationary maps with potential.
	
	Feng  et al. \cite{feng2021geometry} introduced the concept of $ \phi_{S,F} $ harmonic map which is critical point of 
	\begin{equation*}
		\begin{split}
			\int_M  F(\frac{|S_u |^2}{4})\d v_g,
		\end{split}
	\end{equation*}
	where $ S_u=\frac{1}{2}|du|^2g-u^{*}h $ and studied its monotonicity formula and stability. Its stress energy tensor is 
	\begin{equation*}
		\begin{split}
			T(\cdot,\cdot) &=F\left(\frac{\left\|S_u\right\|^{2}}{4}\right) g(\cdot,\cdot)-
			F^{\prime}\left(\frac{\left\|S_u\right\|^{2}}{4}\right) h\left(\sigma_{u}(\cdot), \mathrm{d} u(\cdot)\right),	
		\end{split}
	\end{equation*}
	where $ \sigma_u=h(du(\cdot),du(e_i))du(e_i). $
	Its divergence is (c.f. \cite{feng2021geometry})
	\begin{equation*}
		\begin{split}
			\operatorname{div}(T)(X)=h(\tau_{F},du(X)),
		\end{split}
	\end{equation*}
	where $ \tau_{F}(u)=\operatorname{div} (\sigma_{F,u}), \sigma_{F,u}=F^{\prime}\left(\frac{\left\|S_u\right\|^{2}}{4}\right) \sigma_{u}(\cdot). $

	 Han et al. \cite{han2013stability} studies the monotonicity formula for a new functianal related to conformal maps, which is the critical point of 
	\begin{equation*}
		\begin{split}
			\int_M F(\frac{|T_u|^2}{4})\d v_g,
		\end{split}
	\end{equation*} 
	where $ T_u= u^{*}h-\frac{1}{m}|du|^2g.  $ Its stress energy tensor is 
	\begin{equation*}
		\begin{split}
			T(\cdot,\cdot) &=F\left(\frac{\left\|T_u\right\|^{2}}{4}\right) g(\cdot,\cdot)-
			F^{\prime}\left(\frac{\left\|T_u\right\|^{2}}{4}\right) h\left(\sigma_{u}(\cdot), \mathrm{d} u(\cdot)\right).	
		\end{split}
	\end{equation*}
	Its divergence is 
	\begin{equation*}
		\begin{split}
			\operatorname{div}(T)(X)=h(\tau_{F},du(X)),
		\end{split}
	\end{equation*}
	where $ \tau_{F}(u)=\operatorname{div}(\sigma_{F,u}), \sigma_{F,u}=F^{\prime}\left(\frac{\left\|T_u\right\|^{2}}{4}\right) \sigma_{u}(\cdot). $

	Han et al. \cite{132132132} defined the energy funtional  
	\begin{equation*}
		\begin{split}
			E_{\phi_{S, p, \varepsilon}}(u) &=\int_{M} e_{\phi_{S, p, \varepsilon}}(u) \d v_{g} \\
			&=\int_{M}\left\{\frac{1}{2 p}\left[\frac{m-2 p}{p^{2}}|d u|^{2 p}+m^{\frac{p}{2}-1}\left\|u^{*} h\right\|^{p}\right]+\frac{1}{4 \varepsilon^{n}}\left(1-|u|^{2}\right)^{2}\right\} \d v_{g}.
		\end{split}
	\end{equation*}
	
	Let  $ u:(M,g)\to (N,h) $ , $N$ is endowed with the Kahler structure, insprired by the strong coupling limit of Faddeev-Niemi model(\cite{faddeev1997stable}) ,   Speight and Svesson	\cite{speight2011some,speight2007strong} studied the energy 
	\begin{equation*}
		\begin{split}
			\int_M \frac{|u^{*}\omega |^2}{2}\d v_g.
		\end{split}
	\end{equation*}
Furtherly, Han \cite{Yingbo2015Monotonicity}  derived the monotonicity formula for the $ F $ energy of pull back forms
	\begin{equation}\label{ccc1}
		\begin{split}
			\int_M F(\frac{|u^{*}\omega |^2}{2})\d v_g,
		\end{split}
	\end{equation}
	where $ F $ is the positive function on $ M $, $ \omega $ is  the second fundermental forms of $ (N,h) $

	Branding \cite{Br} introduced  harmonic map with two form and potential  and studied the existentce of corresponding heat flow. Later, We \cite{zbMATH07375745} have studied the existence of harmonic map with two form and potential from compact Riemannian surface with boundary.

	Inspired by the above results, in this paper,  the radial  curvature conditon in  \cite{wei2021dualities} are used to study the monotonicity formula of  generalized map.  We mainly consider the following functional and its several special cases. 
\begin{equation}\label{equ1}
	\begin{split}
		E(u)=\int_{M}\bigg[F\left(\frac{|du|^{2}}{2}\right)+F(\frac{|u^{*}h|^2}{4})+\frac{1}{4\epsilon^n}(1-|u|^2)^2+H(u) \bigg]e^{-\phi}+u^* B\d v_g.
	\end{split}
\end{equation}
where $ B $ is a $ m $ form on $ N $. 	
	
	\begin{defn}Let  $ u:(M^m,g)\to (N^n,h) $.
	We call $ u $ is Ginzburg-Landau type $  \phi $-$ F $ harmonic map coupled with $  \phi $-$ F $ symphonic map with $ m $ form and potential  if  $ u $ solves 	
	\begin{equation}\label{equ2}
		\begin{split}
			&		\delta^\nabla\left( (F^{\prime}\left(\frac{|du|^{2}}{2}\right)du\right) )-F^{\prime}\left(\frac{|du|^{2}}{2}\right)du(\nabla \phi)+\mathrm{\operatorname{div}_g}(\sigma_{F,u})	- F^{\prime}(\frac{|u^{*}h |^2}{4})du(\nabla \phi)\\
			&+\frac{1}{\epsilon^n}(1-|u|^2)u+ Z \left(du(e_1) \wedge \cdots \wedge du(e_m)\right)-\nabla H(u)=0.	
		\end{split}
	\end{equation}
	\end{defn}
We will prove in Lemma \ref{cfd} that the Euler-Lagrange equation of  \eqref{equ1}   is \eqref{equ2}.
For such kinds of map, we will derive its monotonicity formula and Liouville Theorem  in Theorem \ref{thm1} in terms of curvature condition in Theorem 	\ref{kl} and conservation law.

%

Jin \cite{MR1156381} obtained the Liouville theorem  for harmonic map under asymptotical conditions. Later, such conditons were used to study Liouville theorem for generalized harmonic map. In\cite{he2021liouville}, He et al. obtainded  Liouville-Type Theorems for CC-F-Harmonic Maps into a
Carnot Group. In \cite{Han2021}, Han et al. investigated the Liouville Theorems for $ F $-stationary maps under asympototic conditions.
 In \cite{chong2017liouville}, Chong et al. proved Liouville theorem for $ CC $ harmonic map if the map is asymptotically  constant. In Dong \cite[Proposition 4.1, Theorem 5.1]{MR3449358}, Dong obtaied the Liouville type theorems under asymptotic conditions.
   Inspired by the results in Han et al. \cite{132132132}, we  also plan to use the method in \cite{MR1156381} to deal with  Ginzburg-Landau type $\phi$-$F$ harmonic map coupled with $\phi$-$F$ symphonic map which is the solution of  
	\begin{equation}\label{e222}
	\begin{split}
			&		\delta^\nabla\left( (F^{\prime}\left(\frac{|du|^{2}}{2}\right)du\right) )-F^{\prime}\left(\frac{|du|^{2}}{2}\right)du(\nabla \phi)+\mathrm{\operatorname{div}_g}(\sigma_{F,u})	- F^{\prime}(\frac{|u^{*}h |^2}{4})du(\nabla \phi)\\
		&+\frac{1}{\epsilon^n}(1-|u|^2)u=0.
	\end{split}
\end{equation}

We will prove Liouville theorem for such kind of map if the map is asymptotically  constant in Theorem \ref{thmabc}.

In Theorem \ref{thm6.1} and  Theorem \ref{thm6.2}, we will derive Liouville theorem for $ \phi $-$ F $ harmoni  map  using the method  in Zhou\cite{zhou}.

In \cite{MR2959441}, Wang et al. proved Liouville theorem for $ \phi $-harmonic from metric measure space. In 	\cite[Theorem C]{Liu2005}, Liu proved Liouville theorem for $ F $-harmonic map from $ (M,g) $ having slowly divergent 		$ F $-energy  if  $-a^2 \leq K_M \leq 0, \operatorname{Ric}^M(g) \leq-b^2$. However, the cuvature conditions of Theorem 1.3 in these two papers are slightly different from the literature mentioned above, such as \cite{wei2021dualities}\cite{zhao2019monotonicity}.
	\begin{defn}[$\phi$-$V$-$ F $ harmonic map  with potential ]
			we call $ u $ is $ \phi $-$ V $-$ F $ harmonic map   if 
	\begin{equation}\label{e222}
		\begin{split}
			&		\delta^\nabla\left( F^{\prime}\left(\frac{|du|^{2}}{2}\right)du\right) -F^{\prime}\left(\frac{|du|^{2}}{2}\right)du( \nabla \phi-V)=0.
		\end{split}
	\end{equation}
	\end{defn}
	For such kind of map, we will use  the method in  Wang \cite{MR2959441} to deal with $F$ harmonic map from metric measure space in Theorem \ref{thm7}. From our proof, it can be seen that our conditons can also be used to study Liouville theorem for   $\phi$-$F$ harmonic map coupled with $\phi$-$F$ symphonic map.
	
	In \cite{jost-yau}, Jost and Yau introduced Hermitian harmonic map. V harmonic map is the generalization of Hermitian harmonic map.  In the last decades, $ V $ harmonic map has been studied  deeply, such as existence , uniqueness, one can refer to \cite{MR2995205,qiu2017heat}\cite{zbMATH06803427,chen2015maximum}\cite{Chen2020b,Chen2014OmoriYauMP}.
	We define  $\phi$-$V$ harmonic map with  potential which is the generalization of $ V $-harmonic map.
	\begin{defn}
		We call a map is   $ u $ $\phi$-$V$ harmonic map with  potential  if 
		\begin{equation}\label{}
			\begin{split}
				&\tau(u)+du(V- \nabla \phi)-\nabla H(u)=0.
			\end{split}
		\end{equation}	
	\end{defn}		
	For this kind of map, its stress energy tensor is defined by 
	$$S =e(u)g-u^{*}h+H(u)g.$$
	
	For $\phi$-$V$ harmonic map, we will obtain monotonicity formula in (cf. Theorem \ref{thm8.1})
	\begin{defn}
	$ u $ is called $ \phi$-$p $-harmonic map,	if $  u : M \to N $ is the solution of 
		\begin{equation*}
			\begin{split}
				\delta^\nabla \left( |du|^{p-2}du\right) )-|du|^{p-2}du( \nabla \phi)=0.
			\end{split}
		\end{equation*}
	\end{defn}
For such kind of map, we will we use the method and curvature conditions  in \cite{zhao2019monotonicity} to study Liouville theorem of   $ p $ harmonic map (cf. Theorem \ref{thm9.1}). 

	Using the conditons and methods in this paper, one can also consider  more general energy functional on metric measure space, such as 	
\begin{equation*}
	\begin{split}
		E(u)=&\int_{M}\left( F\left(\frac{|du|^{2}}{2}\right)+F(\frac{|u^{*}h|^2}{4})+F(\frac{|T_u|^2}{4})+F(\frac{|S_u |^2}{4})+\frac{1}{4\epsilon^n}(1-|u|^2)^2\right)e^{-\phi} \\
		&+u^* B+H(u)\d v_g,
	\end{split}
\end{equation*}
	where $ S_u=\frac{1}{2}|du|^2g-u^{*}h,T_u= u^{*}h-\frac{1}{m}|du|^2g. $  We leave it to the interested readers as an exercise considering the esstential technique is almost  the same .
	
		Using the conditons and methods in this paper, one can also consider the energy funtional \eqref{ccc1}. 
	
	This paper is orgalized as follows: In section \ref{sec2}, we firstly give some Lemmas which will be used in this paper.  In section \ref{sec3}, we will obtain Liouville theorem for  $\phi$-$F$ harmonic map coupled with $\phi$-$F$ symphonic map with $m$ form and potential from metric measure space under asymptotic condition. In section \ref{sec4}, we will derive Liouville Theorem for Ginzburg-Landau type $\phi$-$F$ harmonic map coupled with $\phi$-$F$ symphonic map from metric measure space under asymptotic condition. In section \ref{sec5}, we will use the method  in Zhou\cite{zhou} to deal with $\phi$-$F$ harmonic map. In section \ref{sec6}, using the method in  Wang \cite{MR2959441}, we will derive  Liouville theorem $\phi$-$F$-$V$-harmonic map from metric measure space. In section \ref{sec7}, we will obtain monotonicity formula for $ \phi $-$ V $ harmonic map with potential. In section \ref{sec8},  using the method in Zhao \cite{zhao2019monotonicity}, we will obtain monotonicity formula for $ \phi $-$ p $ harmonic map from metric measure space.
	
	Throughout this paper,  we use these quantities (cf. \cite[page 3]{MR3449358}),
\begin{equation*}
	\begin{split}
		d_{F}=\sup _{t \geq 0} \frac{t F^{\prime}(t)}{F(t)}\\
		l_{F}=\inf _{t \geq 0} \frac{t F^{\prime}(t)}{F(t)}.
	\end{split}
\end{equation*}

	\section{Preliminary}\label{sec2}
	%
	%
	%
	%
	
	\begin{lem}[cf. \cite{zhao2019monotonicity} ]\label{56}
		
		Let $(M, g) $be a complete Riemannian manifold with a pole $x_0$.
		
		(1) If $-\alpha^2 \leq K_r \leq -\beta^2$ with $ \alpha > 0, \beta > 0 $, then
		$$\beta \coth(\beta r)[g - dr \otimes dr] \leq \operatorname{Hess}(r)\leq  \alpha coth(\alpha r)[g - dr \otimes dr].$$
		
		(2) If 
		$-\frac{A}{(1+r^2)^{1+\epsilon}}\leq K_r\leq \frac{B}{(1+r^2)^{1+\epsilon}}$
		with $ \epsilon >0,A \geq 0 $ and $ 0 \leq B < 2\epsilon $, then
		$$ \frac{1-\frac{B}{2\epsilon}}{r}[g - dr \otimes dr] \leq  \operatorname{Hess}(r) \leq \frac{e^{\dfrac{A}{2\epsilon}}}{r} [g - dr \otimes dr].$$
		
		(3) If $ -\frac{a^2}{c^2+r^2}\leq K_r \leq
		\frac{b^2}{c^2+r^2} $ with $ a_2 \geq 0, 0 \leq b_2\leq \frac{1}{4} $ and $ c_2 \geq  0 $, then
		$$\frac{1+\sqrt{1 - 4b^2}}{2r}[g - dr \otimes dr]\leq  \operatorname{Hess}(r) \leq 1 +
		\frac{1+\sqrt{1 + 4a^2}}{2r}
		[g - dr \otimes dr]. $$
	\end{lem}
	\begin{lem}[cf. \cite{wei2021dualities} ]\label{55}
		Let $(M, g) $be a complete Riemannian manifold with a pole $x_0$.
		
		1) If $-\frac{A(A-1)}{r^2} \leq K_r \leq -\frac{A_1(A_1-1)}{r^2}$ with $A\geq A_1\geq 1 $, then
		
		$$\frac{A_1}{r}[g - dr \otimes dr] \leq \operatorname{Hess}(r)\leq  \frac{A}{r}[g - dr \otimes dr].$$
		
		2) If $-\frac{A}{r^2} \leq K_r \leq -\frac{A_1}{r^2}$ with $A_1\leq A $, then
		
		$$\frac{1+\sqrt{1+4A_1}}{r^2}[g - dr \otimes dr] \leq \operatorname{Hess}(r)\leq  \frac{1+\sqrt{1+4A}}{2r}[g - dr \otimes dr].$$
		
		3) If $\frac{B_1}{r^2} \leq K_r \leq -\frac{B}{r^2}$ with $B_1\leq B\leq \frac{1}{4} $ on , then
		
		$$\frac{1+\sqrt{1-4B}}{2r}[g - dr \otimes dr] \leq \operatorname{Hess}(r)\leq  \frac{1+\sqrt{1+4B_1}}{2r}[g - dr \otimes dr].$$
		
		4) If $\frac{B_1(1-B_1)}{r^2} \leq K_r \leq -\frac{B(1-B)}{r^2}$ with $B_1, B\leq 1 $ , then
		
		$$\frac{|B-\frac{1}{2}|+\frac{1}{2}}{r}[g - dr \otimes dr] \leq \operatorname{Hess}(r)\leq  \frac{1+\sqrt{1+4B_1(1-B_1)}}{2r}[g - dr \otimes dr].$$
	\end{lem}	

	%

%
%
%

	\begin{lem}[The first varation formula]
		Let $ u:  (M,g)\to (N,h) $ be a smooth map, and $  u_t:  (M,g)\to (N,h) $ is the smooth deformation such that $ u_0=u, V=\frac{\partial u_t}{\partial t}|_{t=0} $, then
		\begin{equation*}
			\begin{split}
				&\left.\frac{d}{d t}\right|_{t=0}\int_{M}\bigg(F\left(\frac{|du|^{2}}{2}\right)+u^* B+H(u)\bigg)e^{-\phi} \d v_g,\\
				&=\int_M\bigg\langle 	\delta^\nabla\left( F^{\prime}\left(\frac{|du|^{2}}{2}\right)du\right) -F^{\prime}\left(\frac{|du|^{2}}{2}\right)du( \nabla \phi)- Z_{x}\left(\xi_{1} \wedge \cdots \wedge \xi_{k}\right)-\nabla H,V \bigg\rangle e^{-\phi} \d v_g.
			\end{split}
		\end{equation*}

	\end{lem}	
	
\begin{proof}

The proof is done by  divergence theorem. Let $\Psi: (-\delta,\delta)\times M \to N $ defined by $ \Psi(x,t)=u_t(x) $. It is well known that we can extend the vector field $ \frac{\partial}{\partial t} $ on $ (-\delta,\delta) $  and vector field $ X $ on $ M $ to $ (-\delta,\delta)\times M$ .	By \cite{Shuxiang2014Liouville},  we define a vector field $ X_t $ by  $g(X_t,Y)=\left\langle d\Psi(\frac{\partial}{\partial t}), du(Y)\right\rangle,$ then we get 
	\begin{equation}\label{}
		\begin{split}
			&\frac{d}{dt}\int_{M}F\left(\frac{|d\Psi|^{2}}{2}\right)e^{-\phi}\d v_{g}+\frac{d}{dt}\int_M H(\Psi)e^{-\phi} \d\nu_{g}\\
				=&\int_{M}F'(\frac{|d\Psi|^{2}}{2})\left(\frac{1}{2}\frac{d}{dt}h(d\Psi(e_i,e_i))\right)e^{-\phi}d\nu_{g}+\int_M \langle\nabla H(\Psi),d\Psi(\frac{\partial}{\partial t})\rangle e^{-\phi}\d v_{g}\\
			=&\int_{M}F'(\frac{|du|^{2}}{2})\left\langle \tilde{\nabla}_{e_i}d\Psi(\frac{\partial}{\partial t}) ,du(e_i))\right\rangle e^{-\phi}\d v_{g}+\int_M \langle\nabla H(u),d\Psi(\frac{\partial}{\partial t})\rangle e^{-\phi}\d v_{g}\\
			=&\int_{M}F'(\frac{|du|^{2}}{2}) {e_i}\left\langle d\Psi(\frac{\partial}{\partial t}) ,du(e_i))\right\rangle-F'(\frac{|du|^{2}}{2})\left\langle d\Psi(\frac{\partial}{\partial t}) ,\tilde{\nabla}_{e_i} d\Psi(e_i))\right\rangle  e^{-\phi}\d v_{g}\\
			&+\int_M \langle\nabla H(u),d\Psi(\frac{\partial}{\partial t})\rangle  e^{-\phi}\d v_g\\
				=&\int_{M}F'(\frac{|du|^{2}}{2}) {e_i}g(X_t,e_i)-F'(\frac{|du|^{2}}{2})\left\langle d\Psi(\frac{\partial}{\partial t}) ,\tilde{\nabla}_{e_i} d\Psi(e_i))\right\rangle  e^{-\phi}\d v_g\\
			&+\int_M \langle\nabla H(u),d\Psi(\frac{\partial}{\partial t})\rangle e^{-\phi}\d v_g\\
					=&\int_{M}F'(\frac{|du|^{2}}{2})\operatorname{div}(X_t)-\langle d\Psi(\frac{\partial}{\partial t}) ,fF'(\frac{|du|^{2}}{2})\left( \tilde{\nabla}_{e_i} du(e_i)-du(\nabla_{e_i} e_i )\right) \rangle e^{-\phi} \d v_g\\
			&	+\int_M \langle\nabla H(u),d\Psi(\frac{\partial}{\partial t})\rangle e^{-\phi} \d v_g\\
			=&\int_{M}\operatorname{div}(e^{-\phi}F'(\frac{|du|^{2}}{2})(X_t))d \nu_g-\bigg\langle F'(\frac{|du|^{2}}{2})\left( \tilde{\nabla}_{e_i} du(e_i)-du(\nabla_{e_i} e_i )\right)- \\
			&	+F'(\frac{|du|^{2}}{2})du( \nabla \phi) -du(\nabla F'(\frac{|du|^{2}}{2}) ) 
			+ \nabla H(u),d\Psi(\frac{\partial}{\partial t}) \bigg\rangle e^{-\phi} \d v_g\\
			=&-\int_M\langle 	\delta^\nabla\left( (F^{\prime}\left(\frac{|du|^{2}}{2}\right)du\right) )-F^{\prime}\left(\frac{|du|^{2}}{2}\right)du( \nabla \phi)+\nabla H,d\Psi(\frac{\partial}{\partial t})  \rangle e^{-\phi} \d v_g
		\end{split}
	\end{equation}

\end{proof}
	For  a $(k+1)$-form $\Omega \in \Gamma\left(\Lambda^{k+1} T^{*} N\right)$,  we  define a smooth section of $\operatorname{Hom}\left(\Lambda^{k} T N, TN\right)$  by the equation
$$
\left\langle\eta, Z_{x}\left(\xi_{1} \wedge \cdots \wedge \xi_{k}\right)\right\rangle_g=\Omega_{x}\left(\eta, \xi_{1}, \ldots, \xi_{k}\right), dB=\Omega,
$$
for all $x \in M$ and all $\eta, \xi_{1}, \ldots, \xi_{k} \in T_{x} N$.
	
	\begin{lem}	[c.f. chapter 2 in \cite{koh2008evolution}]
	Let $ u_t:(M^m,g)\to (N^n,h) $ be  a smooth deformation of u such that such that $ u_0=u, v=\frac{\partial u_t}{\partial t}|_{t=0} $  and	$B \in \Gamma\left(\Lambda^{m} T^{*} M\right),$		
		\begin{equation}\label{}
			\begin{split}
				&\left. \frac{d}{d t}\right|_{t=0} \int_{\Sigma}u_{t}^{*} B										=\int_{\Sigma} \Omega\left(v,(d u)^{\underline{m}}\left(\mathrm{vol}_{g}^{\sharp}\right)\right) d \mathrm{vol}_{g},
			\end{split}
		\end{equation}
	where $ (d u)^{\underline{m}}\left(\mathrm{vol}_{g}^{\sharp}\right):=du(e_1)\wedge du(e_2)\cdots du(e_m) $	
	\end{lem}
	\begin{proof}
		Here, we copy the proof in \cite{koh2008evolution} for the  convenience of readers. Let $ \xi_1,\xi_2,\cdots,\xi_m  \in TM $, for $ \forall p \in M,   $ we have 
 			\begin{equation}\label{}
			\begin{split}
				&\left.\frac{d}{d t}\right|_{t=0} \left(u_{t}^{*} B\right)\left(\xi_{1}, \ldots, \xi_{m}\right)\\
				&=\left.\left(\nabla_{v} B\right)\left(\frac{\partial u}{\partial x^{1}}, \ldots, \frac{\partial u}{\partial x^{m}}\right)\right|_{p}+\left.\sum_{r=1}^{k} B\left(\frac{\partial u}{\partial x^{1}}, \ldots, \frac{\nabla v}{\partial x^{r}}, \ldots, \frac{\partial u}{\partial x^{m}}\right)\right|_{p} \\
				&=\left.\Omega\left(v, \frac{\partial u}{\partial x^{1}}, \ldots, \frac{\partial u}{\partial x^{m}}\right)\right|_{p}+\left.\sum_{i=1}^{k}(-1)^{i-1} \xi_{i}\left\{B\left(v, \frac{\partial u}{\partial x^{1}}, \ldots, \frac{\partial u}{\partial x^{m}}\right)\right\}\right|_{p} \\
				&+\left.\sum_{i<j}(-1)^{i+j} B\left(v, d u\left(\left[\xi_{i}, \xi_{j}\right]\right), \frac{\partial u}{\partial x^{1}}, \ldots, \frac{\widehat{\partial u}}{\partial x^{i}}, \ldots, \frac{\widehat{\partial u}}{\partial x^{j}}, \ldots, \frac{\partial u}{\partial x^{m}}\right)\right|_{p},
			\end{split}
		\end{equation}
		where $ \frac{\partial u_{t}}{\partial x^{i}}:=d u_{t}\left(\xi_{i}\right) $ and $ \frac{\nabla v}{\partial x^{i}}:=\nabla_{\xi_{i}} v$, and $ \frac{\widehat{\partial u}}{\partial x^{j}} $ denotes the omission of this entry.
		$$ \left(u_{t}^{*} B\right)_{p}\left(\xi_{1}, \ldots, \xi_{m}\right)=B_{u_{t}(p)}\left(\frac{\partial u_{t}}{\partial x^{1}}(p), \ldots, \frac{\partial u_{t}}{\partial x^{m}}(p)\right), p \in M, $$
			at $t=0,$ by divergence theorem, we have
		\begin{equation}\label{}
			\begin{split}
				&\left. \frac{d}{d t}\right|_{t=0} \int_{\Sigma}u_{t}^{*} B=\int_{\Sigma} \Omega\left(v,(d u)^{\underline{m}}\left(\mathrm{vol}_{g}^{\sharp}\right)\right) d \mathrm{vol}_{g}
			\end{split}
		\end{equation}
	\end{proof}
	%

	\begin{lem}[c.f. \cite{dffdfddf}]
		Let $ u_t:(M^m,g)\to (N^n,h) $ be  a smooth deformation of u such that such that $ u_0=u, V=\frac{\partial u_t}{\partial t}|_{t=0}, $
		\begin{equation*}
			\begin{split}
				\left.\frac{d}{d t}\right|_{t=0} \int_M  F(\frac{|u^{*} h |^2}{4})e^{-\phi} \d v_{g}= - \int_M \left\langle  \mathrm{\operatorname{div}_g}(\sigma_{F,u})- F^{\prime}(\frac{|u^{*}h |^2}{4})du(\nabla \phi),V\right\rangle e^{-\phi}   \d v_g,
			\end{split}
		\end{equation*}		
where $\sigma_{u}(\cdot)=\sum_{i=1}^{m} \left \langle du(\cdot) ,du(e_i) \right \rangle du(e_i),  \sigma_{F,u}(\cdot)=F^{\prime}\left(\frac{\left\|u^{*} h\right\|^{2}}{4}\right)\sum_{i=1}^{m} \left \langle du(\cdot) ,du(e_i) \right \rangle du(e_i).$
	\end{lem}
	
	\begin{lem}[c.f. \cite{dffdfddf}]\label{dffdf}
		Let $\begin{aligned}
			T_1(\cdot,\cdot) &=F\left(\frac{\left\|u^{*} h\right\|^{2}}{4}\right) g(\cdot,\cdot)-
			F^{\prime}\left(\frac{\left\|u^{*} h\right\|^{2}}{4}\right) h\left(\sigma_{u}(\cdot), \mathrm{d} u(\cdot)\right)
		\end{aligned},$ then we have 
		\begin{equation*}
			\begin{split}
				\operatorname{div}(T_1 )(X)=-h(\mathrm{\operatorname{div}_g}(\sigma_{F,u}) , du(X)).
			\end{split}
		\end{equation*}		
	\end{lem}

	\begin{lem}[c.f. \cite{132132132} ]	Let $ u_t:(M^m,g)\to (N^n,h) $ be  a smooth deformation of u such that such that $ u_0=u, V=\frac{\partial u_t}{\partial t}|_{t=0}, $ then
		\begin{equation*}
			\begin{split}
				\left.\frac{d}{d t}\right|_{t=0} \int_M \frac{1}{4\epsilon^n}(1-|u|^2)^2e^{-\phi} \d v_{g}= - \int_M \left\langle  \frac{1}{\epsilon^n}(1-|u|^2)u,V\right\rangle  e^{-\phi} \d v_{g}.
			\end{split}
		\end{equation*}
	\end{lem}

	\begin{lem}[c.f. \cite{132132132} ]\label{d2587}
		\begin{equation*}
			\begin{split}
				\operatorname{div}_g(\frac{1}{4\epsilon^n}(1-|u|^2)^2 g) (X)= -\langle \frac{1}{\epsilon^n}(1-|u|^2)u,du(X)\rangle.
			\end{split}
		\end{equation*}

	\end{lem}
	Combing  the above Lemmas, we have 
	\begin{lem}\label{dfr}
		Let $ u :(M^m,g)\to (N^n,h) $, for the energy funtional,
		\begin{equation}
			\begin{split}
				E(u)=\int_{M}\bigg[F\left(\frac{|du|^{2}}{2}\right)+F(\frac{|u^{*}h|^2}{4})+\frac{1}{4\epsilon^n}(1-|u|^2)^2+H(u) \bigg]e^{-\phi}+u^* B\d v_{g},
			\end{split}
		\end{equation}	
		Its Euler-Lagrange equation is	
		\begin{equation}
			\begin{split}
				&		\delta^\nabla\left( F^{\prime}\left(\frac{|du|^{2}}{2}\right)du\right) -F^{\prime}\left(\frac{|du|^{2}}{2}\right)du( \nabla \phi)+\mathrm{\operatorname{div}_g}(\sigma_{F,u})	- F^{\prime}(\frac{|u^{*}h |^2}{4})du( \nabla \phi)\\
				&+\frac{1}{\epsilon^n}(1-|u|^2)u+ Z \left(du(e_1) \wedge \cdots \wedge du(e_m)\right)-\nabla H(u)=0.	
			\end{split}
		\end{equation}	
	\end{lem}
	
	\section{$\phi$-$F$ harmonic map coupled with $\phi$-$F$ symphonic map with $m$ form and potential from metric measure space }\label{sec3}	
	In this section, we will derive the monotonicity formula of the solution of \eqref{equ2} and get the Liouville theorem.

	

	\begin{lem}\label{lem41}
		For the $\phi$-$F $ stress energy tensor 
		\begin{equation*}
			\begin{split}
				S =F\left(\frac{|du|^{2}}{2}\right) g-F^{\prime}\left(\frac{|du|^{2}}{2}\right)u^{*}h+H(u)g,
			\end{split}
		\end{equation*}
		and
		\begin{equation*}
			\begin{split}
				T &=F\left(\frac{\left\|u^{*} h\right\|^{2}}{4}\right) g-
				F^{\prime}\left(\frac{\left\|u^{*} h\right\|^{2}}{4}\right) h\left(\sigma_{u}(\cdot), \mathrm{d} u(\cdot)\right)+\frac{1}{4\epsilon^n}(1-|u|^2)^2 g .
			\end{split}
		\end{equation*}
		If $ u $ is the solution of \eqref{equ2},	  we  have 
		\begin{equation*}
			\begin{split}
				\operatorname{div} \left(S+T\right)(X)=-&\langle du(X),F^{\prime}\left(\frac{|du|^{2}}{2}\right)du( \nabla \phi)-F^{\prime}\left(\frac{|u^{*}h|^{2}}{4}\right)du( \nabla \phi)\rangle.
			\end{split}
		\end{equation*}
	\end{lem}
	
	\begin{proof}
		By the formular in \cite{2011On},\cite{Mitsunori1999Geometry}
		\begin{equation*}
			\begin{split}
				\left(\operatorname{div} S\right)(X)&=F^{\prime}\left(\frac{|du|^{2}}{2}\right)\left\langle\delta^{\nabla} du, i_{X} du\right\rangle+F^{\prime}\left(\frac{|du|^{2}}{2}\right)\left\langle i_{X} d^{\nabla} du, du\right\rangle\\
				&-\left\langle i_{\operatorname{grad}{\left(F^{\prime}\left(\frac{|du|^{2}}{2}\right)\right)} }du, i_{X} du\right\rangle+\langle \nabla H(u),du(X)\rangle\\
				&=-\langle\tau_{F,u}-\nabla H(u),du(X)\rangle,\\
			\end{split}
		\end{equation*}
		where  $ \tau_{F,u}= 	\delta^\nabla\left( F^{\prime}\left(\frac{|\omega|^{2}}{2}\right)du\right) $. By Lemma \ref{dffdf} and Lemma \ref{d2587}	,	
		\begin{equation*}
			\begin{split}
				(\operatorname{div} T)(X)=-\langle \mathrm{\operatorname{div}_g}(\sigma_{F,u})
				+\frac{1}{\epsilon^n}(1-|u|^2)u  ,du(X)\rangle.
			\end{split}
		\end{equation*}
		Thus, we have
		\begin{equation*}
			\begin{split}
				\operatorname{div} (T+S)(X)=&\langle Z_{x}\left(du(e_1) \wedge \cdots \wedge du(e_m)\right),du(X)\rangle,\\	
				-&\langle du(X),F^{\prime}\left(\frac{|du|^{2}}{2}\right)du( \nabla \phi)-F^{\prime}\left(\frac{|u^{*}h|^{2}}{4}\right)du( \nabla \phi)\rangle,
			\end{split}
		\end{equation*}
		Using the antisymmetry of $ Z $, we finish the proof .
	\end{proof}

		%
		%
		%
	
	\begin{lem}\label{ccc}
		Let $ (M^m, g=\eta^2g_0,e^{-\phi}\mathrm{d} v_g) $ be a complete manifold with a pole $x_0$. Assume that there exists two
		positive functions $h_1(r)$ and $h_2(r)$ such that
		$$h_1(r)[g - dr \otimes dr] \leq \operatorname{Hess}(r)\leq  h_2(r)[g - dr \otimes dr].$$
		If $ rh_2(r)\geq 1,\frac{\partial \log \eta }{\partial r} \geq 0,H \geq 0 $,	then 
		\begin{equation}\label{e1}
			\begin{split}
				\left\langle S+T, \nabla X^{b}\right\rangle
				\geq 	&
				\left( F\left(\frac{|du|^{2}}{2}\right) +\frac{1}{4\epsilon^n}(1-|u|^2)^2+F(\frac{|u^{*}h |^2}{4})+H(u)\right)\\
				&\times \left(1+r \sum_{i=1}^{m-1} rh_1(r)-4r d_F h_2(r)+r\frac{\partial \log \eta }{\partial r} (m-4d_F)\right) .
			\end{split}
		\end{equation}
	where $ S,T $ is as in the Lemma \ref{lem41}.
	\end{lem}
	
\begin{proof}

Let $ S,T $ be as in the Lemma \ref{lem41},	$ X={}^{g_0}\nabla(\frac{1}{2}r^2) ,g=\eta ^2 g_0$, then we have 
	\begin{equation}\label{ppp25}
		\begin{split}
			\nabla  X^{b}
			=r\frac{\partial \log \eta }{\partial r}g+\frac{1}{2}\eta ^2\mathcal{L}_X(g_0).
		\end{split}
	\end{equation}
		Let ${\left\{e_{i}\right\}_{i=1}^{m}}$ be an orthonormal frame with respect to ${g_{0}}$ and ${e_{m}=\frac{\partial}{\partial r}}$. We may assume that  ${\operatorname{Hess}_{g_{0}}\left(r^{2}\right)}$ is a diagonal matrix with respect to ${\left\{e_{i}\right\}}$. Note that ${\left\{\hat{e_{i}}=\eta^{-1} e_{i}\right\}}$ is an orthonormal frame with respect to ${g}$. Therefore	 
	\begin{equation}\label{pl}
		\begin{split}
			& \eta ^2	\left\langle S, \frac{1}{2}\mathcal{L}_X(g_0)\right\rangle_g \\
			&	= \left( F\left(\frac{|du|^{2}}{2}\right) +H\right)\left(1+r \sum_{i=1}^{m-1} \operatorname{Hess}_{g_0}(r)\left(e_{i}, e_{i}\right)\right)-\F\left|du\left(\frac{\partial}{\partial r}\right)\right|^{2}\\& 	-\sum_{i,j=1}^{m-1}\F r \operatorname{Hess}_{g_0}(r)\left(e_{i}, e_{j}\right)\left\langle du\left(e_{i}\right), du\left(e_{j}\right)\right\rangle h,\\
			&\geq  \left( F\left(\frac{|du|^{2}}{2}\right) +H\right)\left(1+r \sum_{i=1}^{m-1} rh_1(r)\right) -\F\left|du\left(\frac{\partial}{\partial r}\right)\right|^{2}\\
			& 	-\sum_{i=1}^{m-1}\F r h_2(r)\left\langle du\left(e_{i}\right), du\left(e_{i}\right)\right\rangle ,\\	&\geq  \left( F\left(\frac{|du|^{2}}{2}\right) +H\right)\left(1+r \sum_{i=1}^{m-1} rh_1(r)\right) -\F\left|du\left(\frac{\partial}{\partial r}\right)\right|^{2}(1-rh_2(r))\\
			& 	-\sum_{i=1}^{m-1}F\left(\frac{|du|^{2}}{2}\right) r d_F h_2(r) ,\\
			&\geq  \left( F\left(\frac{|du|^{2}}{2}\right) +H\right)\left(1+r \sum_{i=1}^{m-1} rh_1(r)\right) -\F\left|du\left(\frac{\partial}{\partial r}\right)\right|^{2}(1-rh_2(r))\\
			& 	-\sum_{i=1}^{m-1}\left( F \left(\frac{|du|^{2}}{2}\right)+H\right)  r d_F h_2(r) ,\\
			&\geq  \left( F\left(\frac{|du|^{2}}{2}\right) +H\right)\left(1+r \sum_{i=1}^{m-1} rh_1(r)- r d_F h_2(r)\right) -\F\left|du\left(\frac{\partial}{\partial r}\right)\right|^{2}(1-rh_2(r))\\
			&\geq  \left( F\left(\frac{|du|^{2}}{2}\right) +H\right)\left(1+r \sum_{i=1}^{m-1} rh_1(r)-2r d_F h_2(r)\right) ,
		\end{split}
	\end{equation}
	and
	\begin{equation*}
		\begin{split}
			r\frac{\partial \log \eta }{\partial r}	\left\langle S, g\right\rangle&= 	r\frac{\partial \log \eta }{\partial r}\left\langle F\left(\frac{|du|^{2}}{2}\right) g-F^{\prime}\left(\frac{|du|^{2}}{2}\right)u^{*}h+H(u)g,g\right\rangle \\ 
			&\geq  r\frac{\partial \log \eta }{\partial r} \left( (m-2d_F) F(\frac{|du |^2}{2})+mH\right)\\
			&\geq  r\frac{\partial \log \eta }{\partial r}(m-2d_F) \left(  F(\frac{|du |^2}{2})+H(u)\right). 
		\end{split}
	\end{equation*}

	Using the formula $ |u^{*}h|^2= \sum_{i=1}^{m}\sum_{j=1}^{m}h(du(e_i),du(e_j))^2=\sum_{i=1}^{m}\sum_{j=1}^{m}h(\sigma_{u}(e_i),du(e_j)) $, as in (\ref{pl}) , it is not hard to get 
	\begin{equation}\label{kl}
		\begin{split}
			& \eta ^2	\left\langle T, \frac{1}{2}\mathcal{L}_X(g_0)\right\rangle_g \\
			&\geq  \left( F\left(\frac{|u^{*}h|^{2}}{4}\right)+\frac{1}{4\epsilon^n}(1-|u|^2)^2\right)\left(1+r \sum_{i=1}^{m-1} h_1(r)-4r d_F h_2(r)\right) 
		\end{split}
	\end{equation}
	and
	\begin{equation*}
		\begin{split}
			r\frac{\partial \log \eta }{\partial r}	\left\langle T, g\right\rangle=& 	r\frac{\partial \log \eta }{\partial r}\left\langle F\left(\frac{\left\|u^{*} h\right\|^{2}}{4}\right) g-
			F^{\prime}\left(\frac{\left\|u^{*} h\right\|^{2}}{4}\right) h\left(\sigma_{u}(\cdot), \mathrm{d} u(\cdot)\right)+\frac{1}{4\epsilon^n}(1-|u|^2)^2 g(\cdot, \cdot),g(\cdot, \cdot)\right\rangle \\ 
			\geq&  r\frac{\partial \log \eta }{\partial r} \left[ (m-4d_F)  \left( F\left(\frac{|u^{*}h|^{2}}{4}\right)+\frac{1}{4\epsilon^n}(1-|u|^2)^2\right)\right]
		\end{split}
	\end{equation*}
	\end{proof}
	\begin{thm}\label{thm1}
		Let $ (M^m, g=\eta^2g_0,e^{-\phi}\mathrm{d} v_g) $ be a complete metric measure space with a pole $  x_0.$  Let   $ (N^n, h) $ be a Riemannian manifold. 
		Let $  u : M \to N $ be  the soluton of \eqref{equ2}. We assume that
		$$ \partial_r\phi \geq 0, H>0,r\frac{\partial \log \eta }{\partial r}\geq 0, d_F \leq \frac{m}{4},  rh_2(r)\geq 1. $$
		Assume that the radial curvature $K_r$	of $M$ satisfies one of the following seven conditions:

		(1)  If $-\frac{A(A-1)}{r^2} \leq K_r \leq -\frac{A_1(A_1-1)}{r^2}$ with $A\geq A_1\geq 1 $, $1+(m-1) A_{1}-4d_F   A>0$;
		
		(2)  If $-\frac{A}{r^2} \leq K_r \leq -\frac{A_1}{r^2}$ with $A_1\leq A $ and  $1+(m-1) \frac{1+\sqrt{1+4 A_{1}}}{2}- 2d_F(1+\sqrt{1+4 A})>0$;
		
		(3)  $\frac{B_1}{r^2} \leq K_r \leq -\frac{B}{r^2}$ with $B_1\leq B\leq \frac{1}{4} $ holds with $1+(m-1)\left(\left|B-\frac{1}{2}\right|+\frac{1}{2}\right)- 2d_F\left(1+\sqrt{1+4 B_{1}\left(1-B_{1}\right)}\right)>0$;
		
		(4) If $\frac{B_1(1-B_1)}{r^2} \leq K_r \leq -\frac{B(1-B)}{r^2}$ with $B_1, B\leq 1 $ on holds with $1+(m-1) \frac{1+\sqrt{1-4 B}}{2}- \left(1+\sqrt{1+4 B_{1}}\right)2d_F>0$;
		
		(5) $-\alpha^{2} \leqslant K(r) \leqslant-\beta^{2}$ with $\alpha>0, \beta>0$ and $(m-1) \beta-  \alpha 4d_{F}  >0$
		
		(6) $K(r)=0$ with $m-4 d_F >0$;
		
		(7)  $-\frac{A}{\left(1+r^{2}\right)^{1+\epsilon}} \leqslant K(r) \leqslant \frac{B}{\left(1+r^{2}\right)^{1+\epsilon}}$ with $\epsilon>0, A \geqslant 0,0<B<2 \epsilon$ and $m-(m-1) \frac{B}{2 \epsilon}-4d_F  \mathrm{e}^{\frac{A}{2 \epsilon}} >0$.
		
 Then, we have 
		\begin{equation}\label{er}
			\begin{split}
				\frac{ \int_{ B_{R_1}(x)}\bigg(F\left(\frac{|du|^{2}}{2}\right)+F(\frac{|u^{*}h|^2}{4})+\frac{1}{4\epsilon^n}(1-|u|^2)^2+H(u)\bigg) e^{-\phi} dv}{R_1^{\sigma}}\\
				\leq \frac{ \int_{ B_{R_2}(x)}\bigg(F\left(\frac{|du|^{2}}{2}\right)+F(\frac{|u^{*}h|^2}{4})+\frac{1}{4\epsilon^n}(1-|u|^2)^2+H(u)\bigg)e^{-\phi}dv}{R_2^{\sigma}},
			\end{split}
		\end{equation}
	where
		\begin{equation*}
			\begin{split}
				\sigma=\begin{cases} 1+(m-1)A_1-4d_{F} A, & \text { if } K(r) \text { satisfies (1), } \\
					1+(m-1) \frac{1+\sqrt{1+4 A_{1}}}{2}-2d_{F} (1+\sqrt{1+4 A}), & \text { if } K(r) \text { satisfies (2), } \\
					1+(m-1)\left(\left|B-\frac{1}{2}\right|+\frac{1}{2}\right)- 2d_{F}\left(1+\sqrt{1+4 B_{1}\left(1-B_{1}\right)}\right), & \text { if } K(r) \text { satisfies (3), } \\
					1+(m-1) \frac{1+\sqrt{1-4 B}}{2}- 2d_{F}\left(1+\sqrt{1+4 B_{1}}\right), & \text { if } K(r) \text { satisfies (4), } \\
					(m-1)\beta-4d_{F}\alpha  , \quad& \text { if } K(r) \text { satisfies (5), } \\
					m-4 d_{F}, & \text { if } K(r) \text { satisfies (6), } \\
					m-(m-1) \frac{B}{2 \epsilon}-4d_{F}  \mathrm{e}^{\frac{A}{2 \epsilon}} , & \text { if } K(r) \text { satisfies (7). }
				\end{cases}
			\end{split}
		\end{equation*}
	\end{thm}
\begin{rem}
	The conditions can also be used to obtain monotonicity formula for the funtional \eqref{ccc1}, see \cite{Yingbo2015Monotonicity}.
\end{rem}
	\begin{rem}
		The conditions (1)-(7) can also be used to deal with other generalized harmonic map, such as $V$ harmonic map, etc. When $F(x)=\frac{\sqrt{2}^p}{p}x^{\frac{p}{2}},$ we  can get monotonicity formular for $ \phi $-$ p $ harmonic map coupled with  $ \phi $-$ p $ symphonic  map with $ m $ form and potential .
	\end{rem}
	
\begin{proof} It is well known that 
	\begin{equation}\label{e2}
		\int_{ B_{R}(x)}	\operatorname{div} \left( i_X(S+T)\right)e^{-\phi}\d v_g =\int_{ B_{R}(x)}\langle (S+T),\nabla 
		X^\sharp \rangle_g e^{-\phi}\d v_g+\int_{ B_{R}(x)}i_X \left( \operatorname{div}(S+T)\right)e^{-\phi}\d v_g ,
	\end{equation}
By divergence theroem, we get 
\begin{equation}\label{k11}
	\begin{split}
		\int_{ B_{R}(x)}	\operatorname{div} \left(e^{-\phi} i_X(S+T)\right) dv_g &=\int_{ B_{R}(x)}\langle (S+T),\nabla 
		X^\sharp \rangle_g e^{-\phi}dv_g-\int_{ B_{R}(x)}(S+T)(X, \nabla \phi) e^{-\phi}dv_g\\
		&+\int_{ B_{R}(x)}i_X \left( \operatorname{div} (S+T)\right) e^{-\phi}dv_g .
	\end{split}
\end{equation}

The LHS of \eqref{k11} is 
\begin{equation}\label{e96}
	\begin{split}
		& \int_{\partial B_{R}(x)} (S+T)\left(X, n \right) dv_g \\
		\leq &R \frac{d}{d R}  \int_{ B_{R}(x)}\left(   F\left(\frac{|du|^{2}}{2}\right)+F(\frac{|u^{*}h|^2}{4})+\frac{1}{4\epsilon^n}(1-|u|^2)^2+H(u) \right) e^{-\phi}\d v_g.
	\end{split}
\end{equation}
The RHS of \eqref{k11} is 
\begin{equation}\label{et}
	\begin{split}
	 &-(S+T)(X, \nabla \phi)+i_X \left( \operatorname{div} (S+T)\right) \\
		= &-(F(\frac{|du |^2}{2})+H)\nabla_{X}e^{-\phi} + F^\prime(\frac{|du |^2}{2}) u^*h(X, \nabla \phi)\\
		&-\left \langle du(X),\tau_{F,H}(u) +\operatorname{div}(\sigma_{F,u})+\frac{1}{\epsilon^n}(1-|u|^2)u\right \rangle \\
		&-\left( F(\frac{|u^* h |^2}{4})+\frac{1}{4\epsilon^n}(1-|u|^2)^2\right) \nabla_{X}e^{-\phi}+ F^\prime(\frac{|du |^2}{2}) u^*h(X, \nabla \phi).
	\end{split}
\end{equation}
So by \eqref{et}, RHS of \eqref{k11} is 
\begin{equation*}
	\begin{split}
		\int_{ B_{R}(x)} \langle S,\nabla 
		X^\sharp \rangle_g e^{-\phi}dv_g-	\int_{ B_{R}(x)}\bigg(F(\frac{|du |^2}{2})+H+F(\frac{|u^* h |^2}{4})+\frac{1}{4\epsilon^n}(1-|u|^2)^2)\left( \nabla_{X}e^{-\phi}\right)\bigg)  e^{-\phi} \d v_g.
	\end{split}
\end{equation*}	
		By \eqref{e1}, we get 	
	\begin{equation}\label{e3}
		\begin{split}
			&\int_{ B_{R}(x)}	\left\langle S+T, \nabla X^{\sharp}\right\rangle  e^{-\phi} \d v_g\\
			\geq 	& \int_{ B_{R}(x)} \left( F\left(\frac{|du|^{2}}{2}\right)+F(\frac{|u^{*}h|^2}{4})+\frac{1}{4\epsilon^n}(1-|u|^2)^2+H(u)\right) e^{-\phi} \d v_g\\
			&\times \left(1+r \sum_{i=1}^{m-1} h_1(r)-4d_Fr  h_2(r)+r\frac{\partial \log \eta }{\partial r} (m-4d_F)\right) .
		\end{split}
	\end{equation}
	By the conditons (1)-(7) and the argument in  \cite{wei2021dualities},  we know that 
	\begin{equation*}
		\begin{split}
			1+r \sum_{i=1}^{m-1} h_1(r)-2 d_F rh_2(r)\geq \sigma.
		\end{split}
	\end{equation*}	
	Plugging \eqref{e96}\eqref{e3} into \eqref{e2}, we get
	\begin{equation}
		\begin{split}
			R \frac{d}{d R}  \int_{ B_{R}(x)} \left( F\left(\frac{|du|^{2}}{2}\right)+F(\frac{|u^{*}h|^2}{4})+\frac{1}{4\epsilon^n}(1-|u|^2)^2+H(u)\right)e^{-\phi} dv_g\\
			\geq \sigma\int_{ B_{R}(x)} \left( F\left(\frac{|du|^{2}}{2}\right)+F(\frac{|u^{*}h|^2}{4})+\frac{1}{4\epsilon^n}(1-|u|^2)^2+H(u)\right)e^{-\phi} dv_g.
		\end{split}
	\end{equation}
		
\end{proof}

	\begin{thm}\label{cnm}
		Under	the same assumptions in  Theorem \ref{thm1},
	Let $  u : M \to N $ be  the soluton of \eqref{equ2}. We assume that
		$$ |\nabla \phi | \leq
		\frac{C}{2r}, $$		
		where $  C < \sigma $ is a constant, and $\sigma$ as in \eqref{er}.
		If the generalized  energy of u satisfies
		$\int_{ B_{R}(x)} \left(  F\left(\frac{|du|^{2}}{2}\right)+F(\frac{|u^{*}h|^2}{4})+\frac{1}{4\epsilon^n}(1-|u|^2)^2+H(u)\right)e^{-\phi}  \d v_g= o(R^{\sigma-C-\frac{m}{2}})  $ as $  R \to \infty, $ then $ u $ is a constant map. In particular,
		if the energy of $  u  $ is finite or moderate divergent, then u is constant.
	\end{thm}

	\begin{rem}
		In fact, under similar assumptions about $ K_r $, our Theorems hold also  for  the critical point the energy funtional 
		\begin{equation*}
			\begin{split}
				E(u)= \int_M\bigg( F\left(\frac{|du|^{2}}{2}\right)&+F(\frac{|u^{*}h|^2}{4})+F(\frac{|T_u|^2}{4})+F(\frac{|S_u |^2}{4})+F(\frac{| u^{*}\omega|^2}{4})\\
				&+\frac{1}{4\epsilon^n}(1-|u|^2)^2+H(u)\bigg)e^{-\phi} +u^* B\d v_{g},
			\end{split}
		\end{equation*}
		where $ S=\frac{1}{2}|du|^2g-u^{*}h,T_u= u^{*}h-\frac{1}{m}|du|^2g, $ $ \omega  $ is the kahler form of N.
		The proof is left to the reader for excersize. 
		
	\end{rem}
	\begin{rem}
		Our Theorem genralized Theorem 3.1,3.2 In \cite{zbMATH07128347}.   Proposition 3.4 in  \cite{Br}.  Theorem 1 is similar to Theorem 4.1 in\cite{afuni2015monotonicity}, However, Our condition is different.
	\end{rem}
	\begin{proof}
		We only need prove $u$ is constant if its energy is moderate divergent. By the same process as in the proof of Theorem , we have \eqref{e96} .  If the energy density $e(u)$ does not vanish identically, then there exists $R_{0}>0$ such that
		\[
		R \int_{B_{R} \left(x_{0}\right)}\left( F\left(\frac{|du|^{2}}{2}\right)+F(\frac{|u^{*}h|^2}{4})+\frac{1}{4\epsilon^n}(1-|u|^2)^2+H(u) \right) e^{-\phi} \d v_g \geq C^{\prime},
		\]
		for $R>R_{0}$, where $C^{\prime}>0$.  , we have
		\[
		\int_{\partial B_{R}\left(x_{0}\right)} \left( F\left(\frac{|du|^{2}}{2}\right)+F(\frac{|u^{*}h|^2}{4})+\frac{1}{4\epsilon^n}(1-|u|^2)^2+H(u) \right) e^{-\phi} \d S \geq \frac{\sigma C{\prime}}{R}.
		\]
 We get
		\begin{equation*}
			\begin{split}
				&\lim _{R \rightarrow \infty} \int_{B_{R}\left(x_{0}\right)} \frac{\left( F\left(\frac{|du|^{2}}{2}\right)+F(\frac{|u^{*}h|^2}{4})+\frac{1}{4\epsilon^n}(1-|u|^2)^2+H(u) \right) e^{-\phi} }{\psi(r(x))}\d v_g \\
				&=\int_{0}^{\infty}\left(\frac{1}{\psi(R)} \int_{\partial B_{R}\left(x_{0}\right)} \left( F\left(\frac{|du|^{2}}{2}\right)+F(\frac{|u^{*}h|^2}{4})+\frac{1}{4\epsilon^n}(1-|u|^2)^2+H(u) \right) e^{-\phi} \d S \right) d R \\
				&\geq\sigma C^{\prime} \int_{0}^{\infty} \frac{d R}{R \psi(R)} \\
				&\geq\sigma C^{\prime} \int_{R_{0}}^{\infty} \frac{d R}{R \psi(R)}=\infty	,
			\end{split}
		\end{equation*}
		which contradicts to  the assumption that  the energy is moderate divergent. Therefore, the map $u$ has to be a constant map.
	\end{proof}

	\begin{thm}
		Under	the same assumptions in  Theorem \ref{thm1}, assume that $ D \subset M $
		is a bounded starlike domain with $C^1$ boundary with the pole $ x_0 $. Let $  u : M \to N $ be  the soluton of \eqref{equ2} with $ u|_{\partial D}= y \in N. $  Then $ u $ is constant on $  D. $
	\end{thm}
	
	\begin{proof}
		
		Choose $X=r \frac{\partial}{\partial r}, r=r_{x_{0}}$. Let $ \nu $ be the unit normal vector field of $ \partial D $ . From the proof of Theorem  \ref{er},  we have
		\begin{equation}\label{}
			\begin{split}
				&\int_{D}\left\{\left\langle S+T, \nabla X^{b}\right\rangle+\left(\operatorname{div} S+T\right)(X)\right\} \\
				\geq&\sigma \int_{D} \left( F\left(\frac{|du|^{2}}{2}\right)+F(\frac{|u^{*}h|^2}{4})+\frac{1}{4\epsilon^n}(1-|u|^2)^2+H(u) \right) e^{-\phi}\d v_g .
			\end{split}
		\end{equation}
		For any $\eta \in T(\partial D)$, we have $d u(\eta)=0$ since $\left.u\right|_{\partial D}$ is constant. Therefore, 
		\begin{equation*}
			\begin{split}
			(S+T)(X, \nu) =&r \left( F\left(\frac{|du|^{2}}{2}\right)+F(\frac{|u^{*}h|^2}{4})+\frac{1}{4\epsilon^n}(1-|u|^2)^2+H(u)\right) \left\langle\frac{\partial}{\partial r}, \nu\right\rangle\\
				-&\left( F^{\prime}(\frac{|du |^2}{2})+\right) r\left\langle d u\left(\frac{\partial}{\partial r}\right), d u(\nu)\right\rangle -F^{\prime}(\frac{|u^{*} h|^2}{4})\langle \sigma_{u}(\frac{\partial}{\partial r})  , du(\nu)\rangle \\
				=&r\left\langle\frac{\partial}{\partial r}, \nu\right\rangle(e(u)-fF^{\prime}(\frac{|du |^2}{2})\langle d u(\nu), d u(\nu)\rangle-F^{\prime}(\frac{|u^{*} h|^2}{4})\langle \sigma_{u}(\nu)  , du(\nu)\rangle)	\\
				=&-r\left\langle\frac{\partial}{\partial r}, \nu\right\rangle e(u) \leq 0,
			\end{split}
		\end{equation*}
		where we have used the fact that 
		\begin{equation*}
			\begin{split}
				\langle \sigma_{u}(\nu)  , du(\nu)\rangle)	=F^{\prime}(\frac{|u^{*} h|^2}{4}) \sum_{i=1}^{m} \left\langle du(e_i),du(\nu)\right\rangle ^2.	
			\end{split}
		\end{equation*}
		Thus  the equation 
		\begin{equation}
			\int_{D}	\operatorname{div} \left( i_X(S+T)\right) =\int_{ D}\langle (S+T),\nabla 
			X^\sharp \rangle_g+\int_{ D}i_X \left( \operatorname{div} (S+T)\right) ,
		\end{equation}
		implies that 
		
		\[
		0 \leq \sigma \int_{D} e(u) \leq 0
		\]
		and then $e(u)=0$ on $D$, which implies that $\left.u\right|_{D} \equiv y \in N.$

	\end{proof}

	\section{Liouville Theorem for Ginzburg-Landau type $\phi$-$F$ harmonic map coupled with $\phi$-$F$ symphonic map from metric measure space under asymptotic condition }\label{sec4}
	In this section, we will consider generalized map without potential under asymptotic conditions .

	\begin{thm}\label{thmabc}	Let $ (M^m, g=f^2g_0,e^{-\phi}\mathrm{d} v_g) $ be a complete metric measure space with a pole $  x_0.$  Let   $ (N^n, h) $ be a Riemannian manifold. 
		Let  $ u: (M,g=f^2 g_0,e^{-\phi}\d v_g)\to (N,h)  $ be  Ginzburg-Landau type $\phi$-$F$ harmonic map coupled with $\phi$-$F$ symphonic map, satisfying $F^{\prime}\left(\frac{|d u|^{2}}{2}\right)$ $<+\infty,F^{\prime}\left(\frac{|u^{*}h|^{2}}{4}\right)$ $<+\infty$  and the $F$-lower degree $l_{F}>0$.
		%
If
		\begin{equation}\label{ball2}
		\begin{split}
				\int_{R}^{\infty} \frac{1}{\bigg[\left( \int_{\partial B(R)} e^{-\phi} f^{m-2}\d v_{g_0}\right) ^{\frac{1}{2}} \bigg]^{\frac{4}{3}}}\d r\geq R^{-\frac{\sigma}{3}}.
		\end{split}
	\end{equation}
	
 Assume that there exists two
				positive functions $h_1(r)$ and $h_2(r)$ such that
				$$h_1(r)[g - dr \otimes dr] \leq \operatorname{Hess}(r)\leq  h_2(r)[g - dr \otimes dr].$$
				Suppose that 
				\begin{equation*}
					\begin{split}
						1+(m-1)rh_1(r)-4rd_Fph_2(r)\geq \sigma>1.  \quad 
					\end{split}
				\end{equation*}
			 We assume that  $ \vol_g(B(R)) = o(R^\sigma), \vol_g(\partial B(R)) \gg \frac{1}{4} $ 
				$$ \partial_r\phi \geq 0 , f\geq 1,r\frac{\partial \log f }{\partial r}\geq 0, d_F \leq \frac{m}{4},  rh_2(r)\geq 1. $$
		
		Furtherly, $ u (x) \to p_0 $ as $ |x|\to \infty $, then $ u $  is a constant.
	\end{thm}
	\begin{proof}
		Here we modify the proof of Dong \cite[Proposition 4.1, Theorem 5.1]{MR3449358} and Han et al. \cite{132132132}. If $ u $ is not a  constant map, then  by the proof of Theorem \ref{thm1},
		\begin{equation}\label{87t}
			\begin{split}
				E(u)\geq C(u)R^{\sigma}, \text{as } \quad R \to \infty.
			\end{split}
		\end{equation} 
	
	 Using the same notations as in  \cite{132132132} or  \cite[Proposition 4.1, Theorem 5.1]{MR3449358}. Choose a local coordinate neighbourhood $ (U, \varphi) $ of $ p_0 $ in $ N^n $, such that
	$  \varphi(p_0) = 0 $. The assumption that $u(x) \rightarrow 0$ as $r(x) \rightarrow \infty$ implies that there exists $R_{1}$ such that for $r(x)>$ $R_{1}, u(x) \in U$, and
	 \[
	 \left(\frac{\partial h_{\alpha \beta(u)}}{\partial u^{\gamma}} y^{\gamma}+2 h_{\alpha \beta}(u)\right) \geq\left(h_{\alpha \beta}(u)\right) \text { for } r(x)>R_{1} .
	 \]
	 For $w \in C_{0}^{2}\left(M^{m} \backslash B\left(R_{1}\right), \varphi(U)\right)$, we consider the variation $u+t w: M^{m} \rightarrow$ $N^{n}$ defined as follows:
	 \[
	 (u+t w)(q)=\left\{\begin{array}{ll}
	 	u(q) & \text { if } q \in B\left(R_{1}\right), \\
	 	\varphi^{-1}[(\varphi(u)+t w)(q)] & \text { if } q \in M^{m} \backslash B\left(R_{1}\right)
	 \end{array}\right.
	 \]

	 ${t}$,  ${u+t \omega: M \rightarrow N^{n}}$  $,(u+t \omega)(q)={u(q),  q \in B\left(R_{1}\right)},{\varphi{ }^{-1}(\varphi(u)+t \omega)(q),  q \in M \backslash B\left(R_{1}\right) }$. Since ${u}$ is  ${F}$-critical point, thus we have 
		\[
		\left.\frac{\mathrm{d}}{\mathrm{d} t}\right|_{t=0} E_{F}(u+t \omega)=0.
		\]
	Computing directly, we have 
	
		\begin{equation*}
			\begin{split}
				&	\int_{M^{m} \backslash B\left(R_{1}\right)}  g_{0}^{i j} F^\prime\left(\frac{|d u|^{2}}{2}\right)\left[2 h_{\alpha \beta}(u) \frac{\partial u^{\alpha}}{\partial x_{i}} \frac{\partial w^{\beta}}{\partial x_{j}}+\frac{\partial h_{\alpha \beta}(u)}{\partial y^{\gamma}} w^{\gamma} \frac{\partial u^{\alpha}}{\partial x_{i}} \frac{\partial u^{\beta}}{\partial x_{j}}\right]  e^{-\phi} f^{m-2}(x) \d v_{g_{0}} \\
				&+{\int_{M^{m} \backslash B\left(R_{1}\right)} g_{0}^{i k} g_{0}^{j l} F^{\prime}\left(\frac{\left\|u^{*} h\right\|^{2}}{4}\right) }
				{\left(2 h_{a \beta}(u) \frac{\partial u^{\alpha}}{\partial x_{i}} \frac{\partial \omega^{\beta}}{\partial x_{j}}+\frac{\partial h_{a \beta}(u)}{\partial y^{\zeta}} \omega^{\zeta} \frac{\partial u^{\alpha}}{\partial x_{i}} \frac{\partial u^{\beta}}{\partial x_{j}}\right)}
				{h_{\gamma \zeta} \frac{\partial u^{\gamma}}{\partial x_{k}} \frac{\partial u^{\zeta}}{\partial x_{l}} e^{-\phi}f^{m-4}(x) \mathrm{d} v_{g_{0}}}\\
				&-\int_{M \backslash B\left(R_{1}\right)} \frac{1}{\varepsilon^{n}}\left(1-|u|^{2}\right) u_{\alpha} \omega_{\alpha} e^{-\phi}\d v_g=0.
			\end{split}
		\end{equation*}
		Thus ,we have 
		\begin{equation*}
			\begin{split}
				&	\int_{M^{m} \backslash B\left(R_{1}\right)}  g_{0}^{i j} g_{0}^{kl}|du |^{-2}F^\prime\left(\frac{|d u|^{2}}{2}\right)\left[2 h_{\alpha \beta}(u) \frac{\partial u^{\alpha}}{\partial x_{i}} \frac{\partial w^{\beta}}{\partial x_{j}}+\frac{\partial h_{\alpha \beta}(u)}{\partial y^{\theta}} w^{\theta} \frac{\partial u^{\alpha}}{\partial x_{i}} \frac{\partial u^{\beta}}{\partial x_{j}}\right]h_{\gamma \xi} \frac{\partial u^{\gamma}}{\partial x_{k}} \frac{\partial u^{\xi}}{\partial x_{l}}   e^{-\phi} f^{m-4}(x) \d v_{g_{0}} \\
				&+{\int_{M^{m} \backslash B\left(R_{1}\right)} g_{0}^{i k} g_{0}^{j l} F^{\prime}\left(\frac{\left\|u^{*} h\right\|^{2}}{4}\right) }
				{\left(2 h_{a \beta}(u) \frac{\partial u^{\alpha}}{\partial x_{i}} \frac{\partial \omega^{\beta}}{\partial x_{j}}+\frac{\partial h_{a \beta}(u)}{\partial y^{\theta}} \omega^{\theta} \frac{\partial u^{\alpha}}{\partial x_{i}} \frac{\partial u^{\beta}}{\partial x_{j}}\right)}
				{h_{\gamma \xi} \frac{\partial u^{\gamma}}{\partial x_{k}} \frac{\partial u^{\xi}}{\partial x_{l}} e^{-\phi}f^{m-4}(x) \mathrm{d} v_{g_{0}}}\\
				&-\int_{M \backslash B\left(R_{1}\right)} \frac{1}{\varepsilon^{n}}\left(1-|u|^{2}\right) u_{\alpha} \omega_{\alpha} e^{-\phi} \d v_g=0.
			\end{split}
		\end{equation*}
	
	For $0<\epsilon \leq 1$, define
	\[
	\varphi_{\epsilon}(t)=\left\{\begin{array}{ll}
		1 & t \leq 1 \\
		1+\frac{1-t}{\epsilon} & 1<t<1+\epsilon \\
		0 & t \geq 1+\epsilon
	\end{array}\right.
	\]
	and choose the Lipschitz function $\phi(r(x))$ in \cite[Proposition 4.1, Theorem 5.1]{MR3449358}to be
	\[
	\psi(r(x))=\varphi_{\epsilon}\left(\frac{r(x)}{R}\right)\left(1-\varphi_{1}\left(\frac{r(x)}{R_{1}}\right)\right), R>2 R_{1}.
	\]
		Take $  \omega= \psi(r)\tilde{u}, \tilde{u}_\alpha=\frac{u_\alpha^2-c_\alpha^2}{u_\alpha}$
		
		\begin{equation*}
			\begin{split}
				&	\int_{M^{m} \backslash B\left(R_{1}\right)}  g_{0}^{i j} g_{0}^{kl}|du |^{-2}F^\prime\left(\frac{|d u|^{2}}{2}\right)\left[2 h_{\alpha \beta}(u) \frac{\partial u^{\alpha}}{\partial x_{i}} \frac{\partial \tilde{u}^{\beta}}{\partial x_{j}}+\frac{\partial h_{\alpha \beta}(u)}{\partial y^{\theta}} \tilde{u}^{\theta} \frac{\partial u^{\alpha}}{\partial x_{i}} \frac{\partial u^{\beta}}{\partial x_{j}}\right]h_{\gamma \xi} \frac{\partial u^{\gamma}}{\partial x_{k}} \frac{\partial u^{\xi}}{\partial x_{l}} e^{-\phi}    f^{m-4}(x) \d v_{g_{0}} \\
				&+{\int_{M^{m} \backslash B\left(R_{1}\right)} g_{0}^{i k} g_{0}^{j l} F^{\prime}\left(\frac{\left\|u^{*} h\right\|^{2}}{4}\right) }
				{\left(2 h_{a \beta}(u) \frac{\partial u^{\alpha}}{\partial x_{i}} \frac{\partial \tilde{u}^{\beta}}{\partial x_{j}}+\frac{\partial h_{a \beta}(u)}{\partial y^{\theta}} \tilde{u}^{\theta} \frac{\partial u^{\alpha}}{\partial x_{i}} \frac{\partial u^{\beta}}{\partial x_{j}}\right)}
				{h_{\gamma \xi} \frac{\partial u^{\gamma}}{\partial x_{k}} \frac{\partial u^{\xi}}{\partial x_{l}}e^{-\phi} f^{m-4}(x) \mathrm{d} v_{g_{0}}}\\
				&-\int_{M \backslash B\left(R_{1}\right)} \frac{1}{\varepsilon^{n}}\left(1-|u|^{2}\right) \left( h_{\alpha\beta}u^{\alpha} \tilde{u}^{\beta}  + \frac{\partial h_{\alpha\beta}}{\partial y^\theta}\tilde{u}^\theta u^\alpha u^\beta \right) e^{-\phi}   \d v_g\\
				&=-\bigg[	\int_{M^{m} \backslash B\left(R_{1}\right)}  g_{0}^{i j} g_{0}^{kl}|du |^{-2}F^\prime\left(\frac{|d u|^{2}}{2}\right)\left[2 h_{\alpha \beta}(u) \frac{\partial u^{\alpha}}{\partial x_{i}}\tilde{u}^{\beta} \frac{\partial \psi  }{\partial x_{j}}\right]h_{\gamma \xi} \frac{\partial u^{\gamma}}{\partial x_{k}} \frac{\partial u^{\xi}}{\partial x_{l}} e^{-\phi}    f^{m-4}(x) \d v_{g_{0}} \\
				&+{\int_{M^{m} \backslash B\left(R_{1}\right)} g_{0}^{i k} g_{0}^{j l} F^{\prime}\left(\frac{\left\|u^{*} h\right\|^{2}}{4}\right) }
				{\left(2 h_{a \beta}(u) \frac{\partial u^{\alpha}}{\partial x_{i}} \tilde{u}^{\beta}\frac{\partial \psi  }{\partial x_{j}}\right)}
				{h_{\gamma \xi} \frac{\partial u^{\gamma}}{\partial x_{k}} \frac{\partial u^{\xi}}{\partial x_{l}}e^{-\phi} f^{m-4}(x) \mathrm{d} v_{g_{0}}}\bigg].
			\end{split}
		\end{equation*}	
		Using the definition of $ \psi  $, let $ \nu $ be the outer normal vector field along $ \partial B(R) $,  we have 
		
		\begin{equation}\label{key}
			\begin{split}
				&	\int_{B(R)\backslash B\left(R_{2}\right)}  g_{0}^{i j} g_{0}^{kl}|du |^{-2}F^\prime\left(\frac{|d u|^{2}}{2}\right)\left[2 h_{\alpha \beta}(u) \frac{\partial u^{\alpha}}{\partial x_{i}} \frac{\partial \tilde{u}^{\beta}}{\partial x_{j}}+\frac{\partial h_{\alpha \beta}(u)}{\partial y^{\theta}} \tilde{u}^{\theta} \frac{\partial u^{\alpha}}{\partial x_{i}} \frac{\partial u^{\beta}}{\partial x_{j}}\right]h_{\gamma \xi} \frac{\partial u^{\gamma}}{\partial x_{k}} \frac{\partial u^{\xi}}{\partial x_{l}} \psi    f^{m-4}(x) \d v_{g_{0}} \\
				&+{\int_{B(R) \backslash B\left(R_{2}\right)} g_{0}^{i k} g_{0}^{j l} F^{\prime}\left(\frac{\left\|u^{*} h\right\|^{2}}{4}\right) }
				{\left(2 h_{a \beta}(u) \frac{\partial u^{\alpha}}{\partial x_{i}} \frac{\partial \tilde{u}^{\beta}}{\partial x_{j}}+\frac{\partial h_{a \beta}(u)}{\partial y^{\theta}} \tilde{u}^{\theta} \frac{\partial u^{\alpha}}{\partial x_{i}} \frac{\partial u^{\beta}}{\partial x_{j}}\right)}
				{h_{\gamma \xi} \frac{\partial u^{\gamma}}{\partial x_{k}} \frac{\partial u^{\xi}}{\partial x_{l}}\psi  e^{-\phi} f^{m-4}(x) \mathrm{d} v_{g_{0}}}\\
				&-\int_{B(R) \backslash B\left(R_{2}\right)} \frac{1}{\varepsilon^{n}}\left(1-|u|^{2}\right) \left( h_{\alpha\beta}u^{\alpha} \tilde{u}^{\beta}  + \frac{\partial h_{\alpha\beta}}{\partial y^\theta}\tilde{u}^\theta u^\alpha u^\beta \right) \psi  e^{-\phi}   \d v_g+D(R_1)\\
				&=\bigg[	\int_{\partial B(R)}   g_{0}^{kl}|du |^{-2}F^\prime\left(\frac{|d u|^{2}}{2}\right)\left[2 h_{\alpha \beta}(u) \frac{\partial u^{\alpha}}{\partial x_{i}}\tilde{u}^{\beta}g_0^{ij}\frac{\partial r}{\partial x_j }\right]h_{\gamma \xi} \frac{\partial u^{\gamma}}{\partial x_{k}} \frac{\partial u^{\xi}}{\partial x_{l}}  e^{-\phi}    f^{m-4}(x) \d v_{g_{0}} \\
				&+{\int_{\partial B(R)} g_{0}^{i k}  F^{\prime}\left(\frac{\left\|u^{*} h\right\|^{2}}{4}\right) }
				{\left(2 h_{a \beta}(u) \frac{\partial u^{\alpha}}{\partial x_{i}} \tilde{u}^{\beta}g_0^{jl}\frac{\partial r}{\partial x_j }\right)}
				{h_{\gamma \xi} \frac{\partial u^{\gamma}}{\partial x_{k}} \frac{\partial u^{\xi}}{\partial x_{l}} e^{-\phi} f^{m-4}(x) \mathrm{d} v_{g_{0}}}\bigg]\\
				&=\bigg[	\int_{\partial B(R)}   g_{0}^{kl}|du |^{-2}F^\prime\left(\frac{|d u|^{2}}{2}\right)\left[2 h_{\alpha \beta}(u) \frac{\partial u^{\alpha}}{\partial x_{i}}\tilde{u}^{\beta}\nu^{j}\right]h_{\gamma \xi} \frac{\partial u^{\gamma}}{\partial x_{k}} \frac{\partial u^{\xi}}{\partial x_{l}}  e^{-\phi}    f^{m-4}(x) \d v_{g_{0}} \\
				&+{\int_{\partial B(R)} g_{0}^{i k}  F^{\prime}\left(\frac{\left\|u^{*} h\right\|^{2}}{4}\right) }
				{\left(2 h_{a \beta}(u) \frac{\partial u^{\alpha}}{\partial x_{i}} \tilde{u}^{\beta}\nu^{l}\right)}
				{h_{\gamma \xi} \frac{\partial u^{\gamma}}{\partial x_{k}} \frac{\partial u^{\xi}}{\partial x_{l}} e^{-\phi} f^{m-4}(x) \mathrm{d} v_{g_{0}}}\bigg],
			\end{split}
		\end{equation}		
		where 
		\begin{equation*}
			\begin{split}
				&D(R_1)\\
				=&		\int_{B(R_2) \backslash B\left(R_{1}\right)}  g_{0}^{i j} g_{0}^{kl}|du |^{-2}F^\prime\left(\frac{|d u|^{2}}{2}\right)\left[2 h_{\alpha \beta}(u) \frac{\partial u^{\alpha}}{\partial x_{i}} \frac{\partial \tilde{u}^{\beta}}{\partial x_{j}}+\frac{\partial h_{\alpha \beta}(u)}{\partial y^{\theta}} \tilde{u}^{\theta} \frac{\partial u^{\alpha}}{\partial x_{i}} \frac{\partial u^{\beta}}{\partial x_{j}}\right]h_{\gamma \xi} \frac{\partial u^{\gamma}}{\partial x_{k}} \frac{\partial u^{\xi}}{\partial x_{l}} e^{-\phi}    f^{m-4}(x) \d v_{g_{0}} \\
				&+{\int_{B(R_2) \backslash B\left(R_{1}\right)} g_{0}^{i k} g_{0}^{j l} F^{\prime}\left(\frac{\left\|u^{*} h\right\|^{2}}{4}\right) }
				{\left(2 h_{a \beta}(u) \frac{\partial u^{\alpha}}{\partial x_{i}} \frac{\partial \tilde{u}^{\beta}}{\partial x_{j}}+\frac{\partial h_{a \beta}(u)}{\partial y^{\theta}} \tilde{u}^{\theta} \frac{\partial u^{\alpha}}{\partial x_{i}} \frac{\partial u^{\beta}}{\partial x_{j}}\right)}
				{h_{\gamma \xi} \frac{\partial u^{\gamma}}{\partial x_{k}} \frac{\partial u^{\xi}}{\partial x_{l}}e^{-\phi} f^{m-4}(x) \mathrm{d} v_{g_{0}}}\\
				&-\int_{B(R_2) \backslash B\left(R_{1}\right)} \frac{1}{\varepsilon^{n}}\left(1-|u|^{2}\right) \left( h_{\alpha\beta}u^{\alpha} \tilde{u}^{\beta}  + \frac{\partial h_{\alpha\beta}}{\partial y^\theta}\tilde{u}^\theta u^\alpha u^\beta \right) e^{-\phi}   \d v_g\\
				&-\bigg[	\int_{M^{m} \backslash B\left(R_{1}\right)}  g_{0}^{i j} g_{0}^{kl}|du |^{-2}F^\prime\left(\frac{|d u|^{2}}{2}\right)\left[2 h_{\alpha \beta}(u) \frac{\partial u^{\alpha}}{\partial x_{i}}\tilde{u}^{\beta} \frac{\partial \varphi_1 (\frac{r}{R_1})}{\partial x_{j}}\right]h_{\gamma \xi} \frac{\partial u^{\gamma}}{\partial x_{k}} \frac{\partial u^{\xi}}{\partial x_{l}} e^{-\phi}    f^{m-4}(x) \d v_{g_{0}} \\
				&+{\int_{M^{m} \backslash B\left(R_{1}\right)} g_{0}^{i k} g_{0}^{j l} F^{\prime}\left(\frac{\left\|u^{*} h\right\|^{2}}{4}\right) }
				{\left(2 h_{a \beta}(u) \frac{\partial u^{\alpha}}{\partial x_{i}} \tilde{u}^{\beta}\frac{\partial \varphi_1 (\frac{r}{R_1})}{\partial x_{j}}\right)}
				{h_{\gamma \xi} \frac{\partial u^{\gamma}}{\partial x_{k}} \frac{\partial u^{\xi}}{\partial x_{l}}e^{-\phi} f^{m-4}(x) \mathrm{d} v_{g_{0}}}\bigg].
			\end{split}
		\end{equation*}
		%
				The LHS of \eqref{key} can be estimates as 		
		\begin{equation*}
			\begin{split}
				&	\int_{B(R)\backslash B\left(R_{2}\right)}  g_{0}^{i j} g_{0}^{kl}|du |^{-2}F^\prime\left(\frac{|d u|^{2}}{2}\right)\left[2 h_{\alpha \beta}(u) \frac{\partial u^{\alpha}}{\partial x_{i}} \frac{\partial \tilde{u}^{\beta}}{\partial x_{j}}+\frac{\partial h_{\alpha \beta}(u)}{\partial y^{\theta}} \tilde{u}^{\theta} \frac{\partial u^{\alpha}}{\partial x_{i}} \frac{\partial u^{\beta}}{\partial x_{j}}\right]h_{\gamma \xi} \frac{\partial u^{\gamma}}{\partial x_{k}} \frac{\partial u^{\xi}}{\partial x_{l}} e^{-\phi}    f^{m-4}(x) \d v_{g_{0}} \\
				&+{\int_{B(R) \backslash B\left(R_{2}\right)} g_{0}^{i k} g_{0}^{j l} F^{\prime}\left(\frac{\left\|u^{*} h\right\|^{2}}{4}\right) }
				{\left(2 h_{a \beta}(u) \frac{\partial u^{\alpha}}{\partial x_{i}} \frac{\partial \tilde{u}^{\beta}}{\partial x_{j}}+\frac{\partial h_{a \beta}(u)}{\partial y^{\theta}} \tilde{u}^{\theta} \frac{\partial u^{\alpha}}{\partial x_{i}} \frac{\partial u^{\beta}}{\partial x_{j}}\right)}
				{h_{\gamma \xi} \frac{\partial u^{\gamma}}{\partial x_{k}} \frac{\partial u^{\xi}}{\partial x_{l}}e^{-\phi} f^{m-4}(x) \mathrm{d} v_{g_{0}}}\\
				&-\int_{B(R) \backslash B\left(R_{2}\right)} \frac{1}{\varepsilon^{n}}\left(1-|u|^{2}\right) \left( h_{\alpha\beta}u^{\alpha} \tilde{u}^{\beta}  + \frac{\partial h_{\alpha\beta}}{\partial y^\theta}\tilde{u}^\theta u^\alpha u^\beta \right) e^{-\phi}   \d v_g+D(R_1)\\
				\geq &Z(R)-\int_{B(R) \backslash B\left(R_{2}\right)} \frac{1}{\varepsilon^{n}}\left(1-|u|^{2}\right) \left( h_{\alpha\beta}u^{\alpha} \tilde{u}^{\beta}  + \frac{\partial h_{\alpha\beta}}{\partial y^\theta}\tilde{u}^\theta u^\alpha u^\beta \right) e^{-\phi}   \d v_g,
			\end{split}
		\end{equation*}
		where \begin{equation*}
			\begin{split}
				Z(R)=&	\int_{B(R)\backslash B\left(R_{2}\right)}  g_{0}^{i j} g_{0}^{kl}|du |^{-2}F^\prime\left(\frac{|d u|^{2}}{2}\right)\left[ h_{\alpha \beta}(u) \frac{\partial u^{\alpha}}{\partial x_{i}} \frac{\partial \tilde{u}^{\beta}}{\partial x_{j}}\right]h_{\gamma \xi} \frac{\partial u^{\gamma}}{\partial x_{k}} \frac{\partial u^{\xi}}{\partial x_{l}} e^{-\phi}    f^{m-4}(x) \d v_{g_{0}} \\
				&+{\int_{B(R) \backslash B\left(R_{2}\right)} g_{0}^{i k} g_{0}^{j l} F^{\prime}\left(\frac{\left\|u^{*} h\right\|^{2}}{4}\right) }
				{\left(2 h_{a \beta}(u) \frac{\partial u^{\alpha}}{\partial x_{i}} \frac{\partial \tilde{u}^{\beta}}{\partial x_{j}}\right)}
				{h_{\gamma \xi} \frac{\partial u^{\gamma}}{\partial x_{k}} \frac{\partial u^{\xi}}{\partial x_{l}}e^{-\phi} f^{m-4}(x) \mathrm{d} v_{g_{0}}}+D(R_1)\\
				&\geq  l_{F}\int_{B(R)\backslash B\left(R_{2}\right)}  F(\frac{|du |^2}{2}) dv_g+l_{F}\int_{B(R)\backslash B\left(R_{2}\right)}  F(\frac{|u^{*}h|^2}{4}) dv_g+D(R_1)
			\end{split}
		\end{equation*}		
		The RHS of \eqref{key} can be estimated as follows:  by Han \cite[(16)]{Han2021}

		\begin{equation*}
			\begin{split}
				&{\int_{\partial B(R)} g_{0}^{i k}  F^{\prime}\left(\frac{\left\|u^{*} h\right\|^{2}}{4}\right) }
				{\left(2 h_{a \beta}(u) \frac{\partial u^{\alpha}}{\partial x_{i}} \tilde{u}^{\beta}\nu^{l}\right)}
				{h_{\gamma \xi} \frac{\partial u^{\gamma}}{\partial x_{k}} \frac{\partial u^{\xi}}{\partial x_{l}}e^{-\phi} f^{m-4}(x) \mathrm{d} v_{g_{0}}}\bigg]\\
				&\leq  \sqrt[4]{m}\left(\int_{\partial B(R)} F^{\prime}\left(\frac{\left\|u^{*} h\right\|^{2}}{4}\right)\left\|u^{*} h\right\|^{2} e^{-\phi}f^{m-4}(x) \mathrm{d} S_{g_{0}}\right)^{\frac{3}{4}} \\
				&\times \left(\int_{\partial B(R)} F^{\prime}\left(\frac{\left\|u^{*} h\right\|^{2}}{4}\right) \right. \left.\left(\sum_{\alpha, \beta=1}^{n} h_{\alpha \beta} u^{\alpha} u^{\beta}\right)^{2} e^{-\phi} f^{m-4}(x) \mathrm{d} S_{g_{0}}\right)^{\frac{1}{4}}.
			\end{split}
		\end{equation*}		
	By Dong\cite[(4.5)]{MR3449358}, we know that

		\begin{equation*}
			\begin{split}
				&\bigg[	\int_{\partial B(R)}   g_{0}^{kl}|du |^{-2}F^{\prime}\left(\frac{|d u|^{2}}{2}\right)\left[2 h_{\alpha \beta}(u) \frac{\partial u^{\alpha}}{\partial x_{i}}\tilde{u}^{\beta}\nu^{j}\right]h_{\gamma \xi} \frac{\partial u^{\gamma}}{\partial x_{k}} \frac{\partial u^{\xi}}{\partial x_{l}} e^{-\phi}    f^{m-4}(x) \d v_{g_{0}} \\
				&\leq  \left(\int_{\partial B(R)} g_0^{ij} F^{\prime}(\frac{|du |^2}{2})h_{\alpha\beta} \frac{\partial u^{\alpha}}{\partial x_{i}} \frac{\partial u^{\beta}}{\partial x_{i}} e^{-\phi} f^{m-2}dv_{g_0} \right)^{\frac{1}{2}} \left(\int_{\partial B(R)} F^{\prime}(\frac{|du |^2}{2})h_{\alpha\beta}u^{\alpha}u^{\beta} e^{-\phi} f^{m-2}dv_{g_0} \right)^{\frac{1}{2}} .
			\end{split}
		\end{equation*}
		
		Combining all these together, we have 		
		\begin{equation*}
			\begin{split}
				&Z(R)-\int_{B(R) \backslash B\left(R_{2}\right)} \frac{1}{\varepsilon^{n}}\left(1-|u|^{2}\right) \left(2 h_{\alpha\beta}u^{\alpha} \tilde{u}^{\beta}  + \frac{\partial h_{\alpha\beta}}{\partial y^\theta}\tilde{u}^\theta u^\alpha u^\beta \right) e^{-\phi}   \d v_g\\
				&\leq  \sqrt[4]{m}\left(Z^{\prime}(R)\right)^{\frac{3}{4}}  \left(\int_{\partial B(R)} F^{\prime}\left(\frac{\left\|u^{*} h\right\|^{2}}{4}\right) \right. \left.\left(\sum_{\alpha \beta=1}^{n} h_{\alpha \beta} u^{\alpha} u^{\beta}\right)^{2} e^{-\phi} f^{m-4}(x) \mathrm{d} s_{g_{0}}\right)^{\frac{1}{4}}\\
				&+ \left(Z^{\prime}(R)\right)^{\frac{1}{2}} \left(\int_{\partial B(R)} F^{\prime}(\frac{|du |^2}{2})h_{\alpha\beta}u^{\alpha}u^{\beta}e^{-\phi} f^{m-2}dv_{g_0} \right)^{\frac{1}{2}} .
			\end{split}
		\end{equation*}		
		If $ Z^{\prime}(R)\leq 1, \sigma>1,$ then 	by \eqref{87t},	 for $ R >R_0,  $
		\begin{equation*}
			\begin{split}
				R-R_0+Z(R_0) \geq Z(R)\geq l_F E(u)-\int_{B(R) \backslash B\left(R_{2}\right)} \frac{1}{\varepsilon^{n}}\left(1-|u|^{2}\right)^2e^{-\phi} \d v_g +D(R_1)\\
				\geq l_F CR^{\sigma}-\int_{B(R) \backslash B\left(R_{2}\right)} \frac{1}{\varepsilon^{n}}\left(1-|u|^{2}\right)^2e^{-\phi} \d v_g +D(R_1),
			\end{split}
		\end{equation*}
		where $ E(u) $ is defined in (\ref{et}). This gives an contradiction. So, we may assume $ Z^{\prime}(R)> 1. $		
		\begin{equation*}
			\begin{split}
				&Z(R)-\int_{B(R) \backslash B\left(R_{2}\right)} \frac{1}{\varepsilon^{n}}\left(1-|u|^{2}\right) \left(2 h_{\alpha\beta}u^{\alpha} \tilde{u}^{\beta}  + \frac{\partial h_{\alpha\beta}}{\partial y^\theta}\tilde{u}^\theta u^\alpha u^\beta \right) e^{-\phi}   \d v_g\\
				&\leq  \sqrt[4]{m}\left(Z^{\prime}(R)\right)^{\frac{3}{4}}  M(R),
			\end{split}
		\end{equation*}
		where \begin{equation}\label{pol1}
			\begin{split}
				M(R)&=\bigg[\left(\int_{\partial B(R)} F^{\prime}\left(\frac{\left\|u^{*} h\right\|^{2}}{4}\right) \right. \left.\left(\sum_{\alpha \beta=1}^{n} h_{\alpha \beta} u^{\alpha} u^{\beta}\right)^{2} e^{-\phi}f^{m-4}(x) \mathrm{d} s_{g_{0}}\right)^{\frac{1}{4}}\\
				&+ \left(\int_{\partial B(R)} F^{\prime}(\frac{|du |^2}{2})h_{\alpha\beta}u^{\alpha}u^{\beta} e^{-\phi}f^{m-2}dv_{g_0} \right)^{\frac{1}{2}} \bigg]\\
				&\leq  C \eta(R)^{\frac{1}{2}} \left[ \left( \int_{\partial B(R)}  e^{-\phi}f^{m-4}dv\right) ^{\frac{1}{4}}+\left( \int_{\partial B(R)} e^{-\phi} f^{m-2}\d v_{g_0}\right) ^{\frac{1}{2}}\right] \\
				&\leq  C \eta(R)^{\frac{1}{2}} \left[ 2\left( \int_{\partial B(R)} e^{-\phi} f^{m-2}\d v_{g_0}\right) ^{\frac{1}{2}}+\frac{1}{2}\right]\\
				& \ll  \eta(R)^{\frac{1}{2}}  2\left( \int_{\partial B(R)}  e^{-\phi}f^{m-2}\d v_{g_0}\right) ^{\frac{1}{2}}
			\end{split}
		\end{equation}
	Hence, we get 
	
		\begin{equation*}
			\begin{split}
				&Z(R)+\int_{B(R) \backslash B\left(R_{2}\right)} \frac{1}{4\varepsilon^{n}}\left(1-|u|^{2}\right)^2 e^{-\phi}   \d v_g\\
				&\leq  \sqrt[4]{m}\left(Z^{\prime}(R)+\int_{ \partial B(R) } \frac{1}{4\varepsilon^{n}}\left(1-|u|^{2}\right)^2 e^{-\phi}   \d v_g\right)^{\frac{3}{4}}  M(R).
			\end{split}
		\end{equation*}
	So
		\begin{equation*}
		\begin{split}
			&\frac{\lambda^{\frac{4}{3}}}{\lambda^{\prime}}  \leq m^{\frac{1}{3}}  M(R)^{\frac{4}{3}}.
		\end{split}
	\end{equation*}
where $ \lambda=Z(R)+\int_{B(R) \backslash B\left(R_{2}\right)}\frac{1}{4\varepsilon^{n}}\left(1-|u|^{2}\right)^2    dv .$

	Notice that if
	\begin{equation*}
		\begin{split}
			\int_{R}^{\infty} \frac{1}{\left( \int_{\partial B(R)}  e^{-\phi}f^{m-2} \d v_{g_0}\right) ^{\frac{2}{3}} }dr\geq R^{-\frac{\sigma}{3}}.
		\end{split}
	\end{equation*}

	Then we can get 
		\begin{equation*}
			\begin{split}
				\int_{R}^{\infty} \frac{1}{M^{\frac{4}{3}}(R) }dr \geq\frac{C}{\eta^{\frac{2}{3}}(R)}  	\int_{R}^{\infty} \frac{1}{\bigg[\left( \int_{\partial B(R)}  e^{-\phi}f^{m-2}dv\right) ^{\frac{1}{2}}+\frac{1}{2} \bigg]^{\frac{4}{3}}}\d r \geq  \frac{C}{\eta^{\frac{2}{3}}(R)}R^{-\sigma},
			\end{split}
		\end{equation*}		
		and
		\begin{equation*}
			\begin{split}
				\lambda^{\frac{-1}{3}}\geq \frac{1}{3\sqrt[4]{m}} 	\int_{R}^{\infty} \frac{1}{M^{\frac{4}{3}}(R) }dr
				\geq  \frac{C}{\eta^{\frac{2}{3}}(R)}R^{-\frac{\sigma}{3}}.
			\end{split}
		\end{equation*}
	It implies that 
	\begin{equation*}
		\begin{split}
			Z(R)+\int_{B(R) \backslash B\left(R_{2}\right)} \frac{1}{4\varepsilon^{n}}\left(1-|u|^{2}\right)^2   \d v_g \leq C\eta(R)^2 R^\sigma.
		\end{split}
	\end{equation*}
However
		
		\begin{equation*}
			\begin{split}
				&Z(R)+\int_{B(R) \backslash B\left(R_{2}\right)} \frac{1}{4\varepsilon^{n}}\left(1-|u|^{2}\right)^2   \d v_g\\
				\geq& l_{F}\int_{B(R)\backslash B\left(R_{2}\right)}  F(\frac{|du |^2}{2}) dv_g+l_{F}\int_{B(R)\backslash B\left(R_{2}\right)}  F(\frac{|u^{*}h|^2}{4}) dv_g\\
				&+\int_{B(R) \backslash B\left(R_{2}\right)} \frac{1}{4\varepsilon^{n}}\left(1-|u|^{2}\right)^2 e^{-\phi}   \d v_g+D(R_1)\\
				\geq & C R^{\sigma}.
			\end{split}
		\end{equation*}
Noticing that $\eta$ converges to zero, this gives a contradiction.

	\end{proof}
Inspired by \cite[Theorem 5.3]{MR3449358} ,
\begin{cor}
		Let $ (M^m, g=\eta^2g_0,e^{-\phi}\mathrm{d} v_g) $ be a complete metric measure space with a pole $  x_0.$  Let   $ (N^n, h) $ be a Riemannian manifold. 
	Let $  u : M \to N $ be   Ginzburg-Landau type $\phi$-$F$ harmonic map coupled with $\phi$-$F$ symphonic map (i.e. solution of \eqref{equ2}) from metric measure space with a pole, satisfying $F^{\prime}\left(\frac{|d u|^{2}}{2}\right)$ $<+\infty,F^{\prime}\left(\frac{|u^{*}h|^{2}}{4}\right)$ $<+\infty$  and the $F$-lower degree $l_{F}>0$. We assume that \eqref{ball2} holds and 
	$$  \partial_r\phi\geq 0, H>0,r\frac{\partial \log \eta }{\partial r}\geq 0, d_F \leq \frac{m}{4},  rh_2(r)\geq 1. $$
	Assume that the radial curvature $K_r$	of $M$  with respect ot $ g_0 $ satisfies one of the following seven conditions:
		
	(1)  If $-\frac{A(A-1)}{r^2} \leq K_r \leq -\frac{A_1(A_1-1)}{r^2}$ with $A\geq A_1\geq 1 $, $1+(m-1) A_{1}-4d_F   A>0$;
	
	(2)  If $-\frac{A}{r^2} \leq K_r \leq -\frac{A_1}{r^2}$ with $A_1\leq A $ and  $1+(m-1) \frac{1+\sqrt{1+4 A_{1}}}{2}- 2d_F(1+\sqrt{1+4 A})>0$;
	
	(3)  $\frac{B_1}{r^2} \leq K_r \leq -\frac{B}{r^2}$ with $B_1\leq B\leq \frac{1}{4} $ holds with $1+(m-1)\left(\left|B-\frac{1}{2}\right|+\frac{1}{2}\right)- 2d_F\left(1+\sqrt{1+4 B_{1}\left(1-B_{1}\right)}\right)>0$;
	
	(4) If $\frac{B_1(1-B_1)}{r^2} \leq K_r \leq -\frac{B(1-B)}{r^2}$ with $B_1, B\leq 1 $ on holds with $1+(m-1) \frac{1+\sqrt{1-4 B}}{2}- \left(1+\sqrt{1+4 B_{1}}\right)2d_F>0$;
	
	(5) $-\alpha^{2} \leqslant K(r) \leqslant-\beta^{2}$ with $\alpha>0, \beta>0$ and $(m-1) \beta-  \alpha 4d_{F}  >0$
	
	(6) $K(r)=0$ with $m-4 d_F >0$;
	
	(7)  $-\frac{A}{\left(1+r^{2}\right)^{1+\epsilon}} \leqslant K(r) \leqslant \frac{B}{\left(1+r^{2}\right)^{1+\epsilon}}$ with $\epsilon>0, A \geqslant 0,0<B<2 \epsilon$ and $m-(m-1) \frac{B}{2 \epsilon}-4d_F  \mathrm{e}^{\frac{A}{2 \epsilon}} >0$.
	
	If $ u (x) \to p_0 $ as $ |x|\to \infty $, then $ u $  is a constant.
\end{cor}
\begin{thm}
Let  $u:\left(M^{m}, f^{2} g_{0},e^{-\phi}\mathrm{d} v_g \right) \rightarrow\left(N^{n}, h\right)$ be Ginzburg-Landau type $\phi$-$F$ harmonic map  coupled with $\phi$-$F$ symphonic map (i.e. solution of \eqref{equ2}), satisfying $F^{\prime}\left(\frac{|d u|^{2}}{2}\right)$ $<+\infty,F^{\prime}\left(\frac{|u^{*}h|^{2}}{4}\right)$ $<+\infty$  and the $F$-lower degree $l_{F}>0$ and Suppose $f$ satisfies 
	\begin{equation*}
		\begin{split}
			1+(m-1)rh_1(r)-4rd_Fph_2(r)\geq \sigma>1.  \quad 
		\end{split}
	\end{equation*} and 
				\begin{equation}\label{pol2}
				\begin{split}
					\int_{\partial B(R)}  f^{m-2}dv \leq CR^{\frac{3}{2}} (\log R)^{\frac{3}{2}}.
				\end{split}
			\end{equation}	
	For any $p \in N^{n}$,  if there is an (nonempty) open neighbourhood $U_{p} \subset N^{n}$, such that the family of open sets $\left\{U_{p} | p \in N^{n}\right\}$ has the following property:
	 for some $p \in N^{n}, u(x) \in U_{p}$ as $r(x) \rightarrow \infty$, then $u$ is a constant map.
\end{thm}
\begin{proof}
	Along the same line as \cite[Theorem 5.4]{MR3449358}, we modify  the proof of Theorem \eqref{thmabc}. As in \cite[Theorem 5.4]{MR3449358},  we can construct a family of coordinate neighbourhoods $\left\{U_{p} | p \in N^{n}\right\}$ as follows: let $\left(U_{p}, \varphi\right)$ be a coordinate system centered at $p$ such that
	\[
	\left(\frac{\partial h_{\alpha \beta(y)}}{\partial y^{\gamma}} y^{\gamma}+2 h_{\alpha \beta}(y)\right) \geq\left(h_{\alpha \beta}(y)\right)  \text { on } U_{p}
	\]
	and
	\[
	h_{\alpha \beta}(y) y^{\alpha} y^{\beta} \leq C_{p}
	\]
	for an arbitrary constant $C_{p}$ depending on $p$.

	If $u:\left(M^{m}, f^{2} g_{0}\right) \rightarrow\left(N^{n}, h\right)$ is a non-constant  $C^{2} $- Ginzburg-Landau type $\phi$-$F$ harmonic map coupled with $\phi$-$F$ symphonic map from metric measure space with a pole, and for some $p \in N^{n}, u(x) \in U_{p}$ as $r(x) \rightarrow \infty$, then $Z(R)>0$ for $R$ large enough , as in  the proof of Theorem \eqref{thmabc}, we can prove  there exists $R_{0}$ such that 
	\begin{equation}\label{pol3}
	\begin{split}
			\lambda^{\frac{-1}{3}}\geq \frac{1}{3\sqrt[4]{m}} 	\int_{R}^{\infty} \frac{1}{M^{\frac{4}{3}}(R) }\d r \text { for } R>R_{3} .
	\end{split}
\end{equation}
	 By \eqref{pol1} and \eqref{pol2}, we have
	\[
	M(R) \ll \left(  \int_{\partial B(R)} f^{m-2}(x) d s_{g_{0}} \right)^{\frac{1}{2}} \ll R^{\frac{3}{4}} (\log R)^{\frac{3}{4}} .
	\]
	Hence,
	\[
	\int_{R}^{\infty} \frac{1}{M^{\frac{4}{3}}(r)} d r \gg \int_{R}^{\infty} \frac{1}{r \log R} d r=\infty
	\]
	which is a contradiction to $(\ref{pol3})$.
\end{proof}

Take $ f=1, \phi=0 $ in Theorem \ref{thmabc},  we get 
\begin{cor}[c.f. Corollary 5.2 in \cite{MR3449358} ]
	 Let $u:\left(M^{m}, g_0,\mathrm{d} v_g \right) \rightarrow\left(N^{n}, h\right)$ be a $C^{2}$  Ginzburg-Landau type $F$ harmonic map coupled with $F$ symphonic map (i.e. solution of \eqref{equ2} with $ \phi $=0) from complete Riemannian manifold  $M^m $ with a pole with $l_{F}>0$ and $F^{\prime}\left(\frac{|d u|^{2}}{2}\right)<+\infty$. Suppose that  radial curvature $ K_r $ of $M^{m}$ with respect ot $ g_0 $ satisfies one of the following two conditions:
	 
	(1) $-\frac{A}{\left(1+r^{2}\right)^{1+\epsilon}} \leq K_{r} \leq \frac{B}{\left(1+r^{2}\right)^{1+\epsilon}}$ with $\epsilon>0, A \geq 0,0 \leq B<2 \epsilon$ and $1+(m-$ 1) $\left(1-\frac{B}{2 \epsilon}\right)-4 d_{F} e^{\frac{A}{2 \epsilon}} \geq 2m-5$

	(2) $-\frac{a^{2}}{1+r^{2}} \leq K_{r} \leq \frac{b^{2}}{1+r^{2}}$ with $a \geq 0, b^{2} \in[0,\frac{1}{4} ]$ and $1+(m-1) \frac{1+\sqrt{1-4 b^{2}}}{2}-$ $2d_{F}\left(1+\sqrt{1+4 a^{2}}\right) \geq 2m A^{\prime}-2A-3$, where $A^{\prime}=\frac{1+\sqrt{1+4 a^{2}}}{2}$.
	If $u(x) \rightarrow p_{0} \in N^{n}$ as $r(x) \rightarrow \infty$, then $u$ is a constant map.
\end{cor}
\begin{proof}
	In the case 1, Corollary 5.2 in \cite{MR3449358}, we know that 
	\begin{equation*}
		\begin{split}
		\vol(\partial B(r))\leq \omega_m e^{\frac{(m-1)A}{2\epsilon}} R^{m-1}
		\end{split}
	\end{equation*}
\begin{equation*}
	\begin{split}
		\left( 	\int_{R}^{\infty} \frac{1}{\left( \vol(\partial B(r))\right) ^{\frac{2}{3}} }dr\right)^{-1}  \leq  \omega_m e^{\frac{(m-1)A}{2\epsilon}} \frac{2m-5}{3} R^{\frac{2m-5}{3}}
	\end{split}
\end{equation*}

In the case 1, Corollary 5.2 in \cite{MR3449358}, we know that
	\begin{equation*}
	\begin{split}
		\vol(\partial B(r))\leq C R^{(m-1)A^\prime}
	\end{split}
\end{equation*} 
\begin{equation*}
	\begin{split}
	\left( 	\int_{R}^{\infty} \frac{1}{\left( \vol(\partial B(r))\right) ^{\frac{2}{3}} }dr\right)^{-1} \leq C \left(\frac{2}{3}A^\prime m-\frac{2}{3}-1 \right)  R^{\frac{2}{3}A^\prime m-\frac{2}{3}-1}
	\end{split}
\end{equation*}

\end{proof}

Hereafter,	in the remained sections in this paper, we will only consider the generalized  harmonic map with potential for simplicity, whose equation doesn't inculde terms about   $ u^*h. $ The essential idea in the proof is the same and can be adapted to other map whose equation and energy funtional maybe have more terms.
	
	\section{Using the method  in Zhou\cite{zhou} to deal with $\phi$-$F$ harmonic map}\label{sec5}
	\begin{thm}\label{thm6.1}
			Let $ (M^m, g=\eta^2g_0,e^{-\phi}\mathrm{d} v_g) $ be a complete metric measure space with a pole $  x_0.$  Let   $ (N^n, h) $ be a Riemannian manifold. 
		If $ u: \left(M^{m}, f^{2} g_{0},e^{-\phi}\mathrm{d} v_g \right)\to (N,h)  $ be  $\phi$-$F$ harmonic map with finite $ F $-energy, assume that there exists two
		positive functions $h_1(r)$ and $h_2(r)$ such that
		$$h_1(r)[g_0 - dr \otimes dr] \leq \operatorname{Hess}_{g_0}(r)\leq  h_2(r)[g_0 - dr \otimes dr] .$$
		and
		\begin{equation*}
			\begin{split}
				r\frac{\partial \log f }{\partial r}  (m-2d_F)+1+r \sum_{i=1}^{m-1} h_1(r) \geq \sigma, \quad |\nabla \phi|\leq \frac{C}{r},
			\end{split}
		\end{equation*}
	where $ r(x)=d_{g_0}(x,x_0) $,			then u is constant .

	\end{thm}
	
	\begin{rem}
		We can establish  similar theroems for critical point of  these energy functionals	\begin{equation*}
			\begin{split}
				\int_M F(\frac{|u^{*} h|^2}{4})+Hdv, \quad \text{or}\quad \int_M  F(\frac{|S_u |^2}{4})+Hdv \quad \text{or}\quad \int_M F(\frac{|T_u|^2}{4})+Hdv.
			\end{split}
		\end{equation*}
		using conservation law, using the method in \cite{feng2021geometry}\cite{han2013stability}\cite{Han2021} .
	\end{rem}

	\begin{proof}
		Following \cite{zhou} \cite{li2012monotonicity}, let $S=F(\frac{|du |^2}{2})g-F^\prime(\frac{|du |^2}{2})u^{*}h+H(u)g $, $ X=\psi(r)\frac{\partial}{\partial r}.$		
			Let ${\left\{e_{i}\right\}_{i=1}^{m}}$ be an orthonormal frame with respect to ${g_{0}}$ and ${e_{m}=\frac{\partial}{\partial r}}$. We may assume that  ${\operatorname{Hess}_{g_{0}}\left(r^{2}\right)}$ is a diagonal matrix with respect to ${\left\{e_{i}\right\}}$. Note that ${\left\{\hat{e_{i}}=\eta^{-1} e_{i}\right\}}$ is an orthonormal frame with respect to ${g}$.
			
		By \cite[(19)(20)(21)]{zhou}, we know that 
		\begin{equation*}
			\begin{split}
				\nabla_{\frac{\partial}{\partial r}} X&= \psi^{\prime}(r) \frac{\partial}{\partial r} , \quad \quad 
				{}^{g_0}\nabla_{e_{s}} X= \psi(r) \nabla_{e_{s}} \frac{\partial}{\partial r}=r \operatorname{Hess}_{g_0}(r)\left(e_{s}, e_{t}\right) e_{t} ; \\
				\left\langle	{}^{g_0}\nabla_{e_{\alpha}} X, e_{\alpha}\right\rangle&= \psi^{\prime}(r)+ \psi(r) \operatorname{Hess}_{g_0}(r)\left(e_{s}, e_{s}\right) .
			\end{split}
		\end{equation*}	
	We have 
\begin{equation*}
	\begin{split}
		\left\langle S, {}^g\nabla X\right\rangle_g=	\eta ^2	\left\langle S, \frac{1}{2}\mathcal{L}_X(g_0)\right\rangle_g +\psi(r)r\frac{\partial \log \eta }{\partial r}	\left\langle S, g\right\rangle_g
	\end{split}
\end{equation*}	
and 	
	\begin{equation}
		\begin{split}
			\psi(r)r\frac{\partial \log \eta }{\partial r}	\left\langle S, g\right\rangle_g
			\geq  \psi(r)r\frac{\partial \log \eta }{\partial r} \left( (m-2d_F) F(\frac{|du |^2}{2})+mH\right). 
		\end{split}
	\end{equation}
	
		So , by \cite[(22)]{zhou}, we have 		

		\begin{equation*}
			\begin{split}
				\left\langle S, {}^{g_0}\nabla X\right\rangle_g
				=& F\left(\frac{|\mathrm{d} u|^{2}}{2}\right)\left[ \psi^{\prime}(r)+ \psi(r) \operatorname{Hess}_{g_0}(r)\left(e_{s}, e_{s}\right)\right] \\
				&-F^{\prime}\left(\frac{|\mathrm{d} u|^{2}}{2}\right)\left\langle\mathrm{d} u e_{s}, \mathrm{~d} u e_{t}\right\rangle  \psi(r) \operatorname{Hess}_{g_0}(r)\left(e_{s}, e_{t}\right) \\
				&- \psi^{\prime}(r) F^{\prime}\left(\frac{|\mathrm{d} u|^{2}}{2}\right)\left\langle\mathrm{d} u \left( \frac{\partial}{\partial r}\right) , \mathrm{~d} u \left( \frac{\partial}{\partial r}\right) \right\rangle.
			\end{split}
		\end{equation*}
		Thus, we have 
		\begin{equation}\label{w8}
			\begin{split}
				& \eta ^2	\left\langle S, \frac{1}{2}\mathcal{L}_X(g_0)\right\rangle_g \\
				&	= \left( F\left(\frac{|du|^{2}}{2}\right) +H\right)\left(\psi^{\prime}(r)+\psi(r)\sum_{i=1}^{m-1} \operatorname{Hess}_{g_0}(r)\left(e_{i}, e_{i}\right)\right)-\psi^{\prime}(r)\F\left|du\left(\frac{\partial}{\partial r}\right)\right|^{2}\\
				& 	-\sum_{i,j=1}^{m-1}\psi(r)\F  \operatorname{Hess}_{g_0}(r)\left(e_{i}, e_{j}\right)\left\langle du\left(e_{i}\right), du\left(e_{j}\right)\right\rangle ,\\
				&\geq  \left( F\left(\frac{|du|^{2}}{2}\right) +H\right)\left(\psi^{\prime}(r)+\psi(r)\sum_{i=1}^{m-1} h_1(r)\right) -\psi^{\prime}(r)\F\left|du\left(\frac{\partial}{\partial r}\right)\right|^{2}\\
				& 	-\psi(r)  h_2(r)\F\sum_{i=1}^{m-1}\left\langle du\left(e_{i}\right), du\left(e_{i}\right)\right\rangle ,\\	
							\end{split}
		\end{equation}
In additon, we have 		
		\begin{equation*}
			\begin{split}
				\left(\operatorname{div} S \right)(X)=\left[ F\left(\frac{|du|^{2}}{2}\right)\right] df(X),
			\end{split}
		\end{equation*}
	and		
	\begin{equation}\label{k53}
		\begin{split}
			\int_{ M}	\operatorname{div} \left(e^{-\phi} i_X(S)\right) dv_g =\int_{ M}\langle S,\nabla 
			X^\sharp \rangle_g e^{-\phi}dv_g-\int_{M}S(X,\nabla \phi) e^{-\phi}dv_g\\
			+\int_{ M}i_X \left( \operatorname{div} (S)\right) e^{-\phi}dv_g .
		\end{split}
	\end{equation}
	As before, we have 	
		\begin{equation}\label{cfd}
			\begin{split}
			-\int_{M}S(X,\nabla \phi) e^{-\phi}dv_g
			+\int_{ M}i_X \left( \operatorname{div} (S)\right) e^{-\phi}dv_g  \\
			\geq	-\bigg(F(\frac{|du |^2}{2})+H)\bigg) \left( \nabla_{X}\phi\right)+F^\prime(\frac{|du |^2}{2})\langle du(X) , du(\nabla \phi)\rangle
			\end{split}
		\end{equation}		
		where $ X $ has compact support,  by\eqref{w8} \eqref{k53}\eqref{cfd}, we get 
		
		\begin{equation*}
			\begin{split}
				&0=\int_{ B_{R}(x)}\langle S,\nabla 
				X^\sharp \rangle_ge^{-\phi}dv_g-\int_{M}S(X,\nabla \phi) e^{-\phi}dv_g
				+\int_{ M}i_X \left( \operatorname{div} (S)\right) e^{-\phi}dv_g \\
				&\geq  \int_{ B_{R}(x)} \left( F\left(\frac{|du|^{2}}{2}\right) +H\right)\left(\psi(r)r\frac{\partial \log \eta }{\partial r}  (m-2d_F)+\psi^{\prime}(r)+\psi(r) \sum_{i=1}^{m-1} h_1(r)\right)\\
				&	-\psi(r)  h_2(r)\F\sum_{i=1}^{m-1}\left\langle du\left(e_{i}\right), du\left(e_{i}\right)\right\rangle-\psi^{\prime}(r)\F\left|du\left(\frac{\partial}{\partial r}\right)\right|^{2}\\
				& +\int_M F^\prime(\frac{|du |^2}{2})\langle du(X) , du(\nabla \phi)\rangle e^{-\phi}dv_g .\\
			\end{split}
		\end{equation*}
		
Let $ \psi(r)=r\lambda(r) $, we have 
		\begin{equation*}
			\begin{split}
				0	&\geq  \int_{ B_{R}(x)} \left( F\left(\frac{|du|^{2}}{2}\right) +H\right)\left( r^2\lambda(r)\frac{\partial \log \eta }{\partial r}  (m-2d_F)+\lambda(r)+r\lambda(r) \sum_{i=1}^{m-1} h_1(r)  + r \lambda^{\prime}(r)\right)\\
				&-r\lambda(r)  h_2(r)\F\sum_{i=1}^{m-1}|du(e_i) |^2-(r\lambda^{\prime}(r)+\lambda(r))\F\left|du\left(\frac{\partial}{\partial r}\right)\right|^{2} dv_g\\
				& +\int_M F^\prime(\frac{|du |^2}{2})r\lambda(r)\langle du(\frac{\partial }{\partial r}) , du(\nabla \phi)\rangle e^{-\phi}dv_g \\
				&\geq  \int_{ B_{R}(x)} \left( F\left(\frac{|du|^{2}}{2}\right) +H\right)\left( \lambda(r)\sigma + r \lambda^{\prime}(r)\right)\\
				&-r\lambda(r)  h_2(r)\F\sum_{i=1}^{m-1}|du(e_i) |^2-(r\lambda^{\prime}(r)+\lambda(r))\F\left|du\left(\frac{\partial}{\partial r}\right)\right|^{2} e^{-\phi}\d v_g\\
				& +\int_M F^\prime(\frac{|du |^2}{2})r\lambda(r)\langle du(\frac{\partial }{\partial r}) , du(\nabla \phi)\rangle e^{-\phi}\d v_g. \\
			\end{split}
		\end{equation*}		
	By the definition of $ \lambda(r) $,we get 
	\begin{equation*}
		\begin{split}
		0	&\geq  \int_{ B_{\frac{R}{2}}(x)} \sigma\left( F\left(\frac{|du|^{2}}{2}\right) +H\right)+\int_{ T_R(x)} \left( F\left(\frac{|du|^{2}}{2}\right) +H\right)\left( \lambda(r)\sigma \right)-C\int_{ T_{R}(x)} \left( F\left(\frac{|du|^{2}}{2}\right) +H\right)\\
		&-\int_{ B_{\frac{R}{2}}(x)} r  h_2(r)\F\sum_{i=1}^{m-1}|du(e_i) |^2-\int_{ T_R(x)} \lambda(r)r  h_2(r)\F\sum_{i=1}^{m-1}|du(e_i) |^2\\
		&-\int_{ B_{\frac{R}{2}}(x)}\F\left|du\left(\frac{\partial}{\partial r}\right)\right|^{2} e^{-\phi}\d v_g-\int_{ T_R(x)}(\lambda(r))\F\left|du\left(\frac{\partial}{\partial r}\right)\right|^{2} e^{-\phi}\d v_g\\
		&-C\int_{ T_R(x)}\F\left|du\left(\frac{\partial}{\partial r}\right)\right|^{2} e^{-\phi}\d v_g \\
		& -\int_{B_{\frac{R}{2}}(x)} F^\prime(\frac{|du |^2}{2})r\|\nabla \phi\|_\infty|du |^2 e^{-\phi}\d v_g-\int_{T_R(x)} F^\prime(\frac{|du |^2}{2})r\|\nabla \phi\|_\infty|du |^2 e^{-\phi}\d v_g.
		\end{split}
	\end{equation*} 
	
		Thus
		\begin{equation*}
			\begin{split}
				0\geq& 	\lim\limits_{R\to \infty}	\int_{ B_{R}(x)}(\sigma-d_F rh_2(r)-2r\|\nabla \phi\|_\infty d_F) \left( F\left(\frac{|du|^{2}}{2}\right) +H\right)e^{-\phi}\d v_g\\	 &-\int_{T_R(x)}(C+2(C+1)d_Frh_2(r)+2Cd_F+2r\|\nabla \phi\|_\infty d_F)F\left(\frac{|du|^{2}}{2}\right)e^{-\phi}\d v_g.
			\end{split}
		\end{equation*}		
		This implies that  $  u  $  is constant.		
	\end{proof}	
\begin{cor}	Let $ (M^m, g=\eta^2g_0,e^{-\phi}\mathrm{d} v_g) $ be a complete metric measure space with a pole $  x_0.$  Let   $ (N^n, h) $ be a Riemannian manifold. 
		If $ u: \left(M^{m}, f^{2} g_{0},e^{-\phi}\mathrm{d} v_g \right)\to (N,h)  $ be  $\phi$-$F$ harmonic map with finite energy, 
	$ K_r $ . 
	\begin{equation*}
		\begin{split}
			r\frac{\partial \log f }{\partial r}  (m-2d_F)+1+r \sum_{i=1}^{m-1} h_1(r) \geq \sigma, \quad  |\nabla \phi|\leq \frac{C}{R}
		\end{split}
	\end{equation*}
 where $ r(x)=d_{g_0}(x,x_0) $	 and 
\begin{equation}\label{253}
	\begin{split}
		h_1(r)=\begin{cases}
			\frac{A_1}{r}& if  \quad (1) \quad \text{holds} \\
			  \frac{1+\sqrt{1+4A_1}}{2r}& if  \quad (2) \quad \text{holds}\\
			   \frac{|B-\frac{1}{2}|+\frac{1}{2}}{r}& if  \quad (3) \quad \text{holds}\\
			   \frac{1+\sqrt{1-4B}}{2r}& if  \quad (4) \quad \text{holds}\\
			   \frac{\beta}{r} & if  \quad (5) \quad \text{holds}\\
			   \frac{1}{r} &if  \quad (6) \quad \text{holds}\\
			    \frac{B}{2r \epsilon} &if  \quad (7) \quad \text{holds},
		\end{cases}
	\end{split}
\end{equation}
	then u is constant .
\end{cor}
\begin{proof}
If 	$ K_r $ satisfies one of the conditions (1)-(7) in Theorem \ref{thm1}, then $ \operatorname{Hess}(r)\leq  h_2(r)[g - dr \otimes dr] $. By \cite{wei2021dualities} or Lemma \eqref{56} and  Lemma \eqref{55} , there exits some constant $C$ such that  $1 \leq  rh_2(r)\leq C. $ 
\end{proof}
		\begin{thm}	\label{thm6.2}Let $ (M^m, g=\eta^2g_0,e^{-\phi}\mathrm{d} v_g) $ be a complete metric measure space with a pole $  x_0.$  Let   $ (N^n, h) $ be a Riemannian manifold. 
		If $ u:\left(M^{m}, \eta^{2} g_{0}, e^{-\phi}\mathrm{d} v_g \right)\to (N,h)  $ be  $\phi$-$F$ harmonic map with finite energy.  Assume that there exists two 	positive functions $h_1(r)$ and $h_2(r)$ such that
		$$h_1(r)[g_0 - dr \otimes dr] \leq \operatorname{Hess}_{g_0}(r)\leq  h_2(r)[g_0 - dr \otimes dr].$$
	where  $ r(x)=d_{g_0}(x,x_0) $	.	Moreover, we assume
		\begin{equation*}
			\begin{split}
				&r\frac{\partial \log \eta }{\partial r}  (m-2d_F)+1+r \sum_{i=1}^{m-1} h_1(r) \geq \sigma,\\
				&\sigma-1-r\|\nabla \phi\|_{\infty}-2d_F\geq 0,\\
				&\frac{1}{2d_F}\left(\frac{\sigma-1}{r} \right)-h_2(r)-\|\nabla \phi\|_{\infty} \geq 0,
			\end{split}
		\end{equation*}
		then u is constant.

	\end{thm}
\begin{proof}

Let $S=F(\frac{|du |^2}{2})g-F^\prime(\frac{|du |^2}{2})u^{*}h+H(u)g $, $ X=\psi(r)\frac{\partial}{\partial r} $,  Let ${\left\{e_{i}\right\}_{i=1}^{m}}$ be an orthonormal frame with respect to ${g_{0}}$ and ${e_{m}=\frac{\partial}{\partial r}}$. We may assume that  ${\operatorname{Hess}_{g_{0}}\left(r^{2}\right)}$ is a diagonal matrix with respect to ${\left\{e_{i}\right\}}$. Note that ${\left\{\hat{e_{i}}=\eta^{-1} e_{i}\right\}}$ is an orthonormal frame with respect to ${g}$.
		\begin{equation}
		\begin{split}
			\psi(r)\frac{\partial \log \eta }{\partial r}	\left\langle S, g\right\rangle
			\geq  \psi(r)\frac{\partial \log \eta }{\partial r} \left( (m-2d_F) F(\frac{|du |^2}{2})+mH\right). 
		\end{split}
	\end{equation}
	By \eqref{w8}, we have 
		\begin{equation}\label{w88}
		\begin{split}
		 &\left\langle S, \nabla X\right\rangle=	\eta ^2	\left\langle S, \frac{1}{2}\mathcal{L}_X(g_0)\right\rangle_g +\psi(r)\frac{\partial \log \eta }{\partial r}	\left\langle S, g\right\rangle\\
					\geq&  \left( F\left(\frac{|du|^{2}}{2}\right) +H\right)\left(\psi^{\prime}(r)+\psi(r)\sum_{i=1}^{m-1} h_1(r)\right) \\
			& 	-\psi(r)  h_2(r)\F\sum_{i=1}^{m-1}|du(e_i) |^2 -\psi^{\prime}(r)\F\left|du\left(\frac{\partial}{\partial r}\right)\right|^{2},\\
			\geq & \left( \frac{1}{2d_F}F^\prime\left(\frac{|du|^{2}}{2}\right) \right)\left(\psi^{\prime}(r)+\psi(r)\sum_{i=1}^{m-1} h_1(r)\right)\left( \sum_{i=1}^{m-1}|du(e_i) |^2+\left|du\left(\frac{\partial}{\partial r}\right)\right|^{2}\right)  \\
			& 	-\psi(r)  h_2(r)\F\sum_{i=1}^{m-1}|du(e_i) |^2 -\psi^{\prime}(r)\F\left|du\left(\frac{\partial}{\partial r}\right)\right|^{2},\\
			&+H\left(\psi^{\prime}(r)+\psi(r)\sum_{i=1}^{m-1} h_1(r)\right)	\\
				\geq &  F^\prime\left(\frac{|du|^{2}}{2}\right) \left[\left( \frac{1}{2d_F}-1\right) \psi^{\prime}(r)+\frac{1}{2d_F}\psi(r)\left( \sum_{i=1}^{m-1} h_1(r)+\frac{\partial \log \eta }{\partial r}  (m-2d_F) \right)\right]  \left|du\left(\frac{\partial}{\partial r}\right)\right|^{2}  \\
				&+  F^\prime\left(\frac{|du|^{2}}{2}\right) \left(\frac{1}{2d_F}\psi^{\prime}(r)+\frac{1}{2d_F}\psi(r)\left( \sum_{i=1}^{m-1} h_1(r)+\frac{\partial \log \eta }{\partial r}  (m-2d_F) \right)-\psi(r)h_2(r)\right) \sum_{i=1}^{m-1}|du(e_i) |^2  \\
				&+\int_M F^\prime(\frac{|du |^2}{2})\langle du(X) , du(\nabla \phi)\rangle e^{-\phi}dv_g \\
				\geq &  F^\prime\left(\frac{|du|^{2}}{2}\right) \left[\left( \frac{1}{2d_F}-1\right) \psi^{\prime}(r)+\frac{1}{2d_F}\psi(r)\left( \frac{\sigma-1 }{r}-\|\nabla \phi\|_{\infty}\right)\right]  \left|du\left(\frac{\partial}{\partial r}\right)\right|^{2}  \\
				&+  F^\prime\left(\frac{|du|^{2}}{2}\right) \left(\frac{1}{2d_F}\psi^{\prime}(r)+\psi(r)\left[ \frac{1}{2d_F}\left(\frac{\sigma-1}{r} \right)-h_2(r)-\|\nabla \phi\|_{\infty}\right] \right) \sum_{i=1}^{m-1}|du(e_i) |^2.  \\
					\end{split}
	\end{equation}
Take $ \psi(r)=r\lambda(r),  $ we get 	
	
	\begin{equation*}
		\begin{split}
			\left\langle S, \nabla X\right\rangle	\geq &  F^\prime\left(\frac{|du|^{2}}{2}\right) \left[- \psi^{\prime}(r)+\frac{1}{2d_F}\psi(r)\left( \frac{\sigma-1 }{r}-\|\nabla \phi\|_{\infty}\right)\right]  \left|du\left(\frac{\partial}{\partial r}\right)\right|^{2}  \\
			&+  F^\prime\left(\frac{|du|^{2}}{2}\right)\frac{1}{2d_F}\psi^{\prime}(r)  |du |^2\\
			&=F^\prime\left(\frac{|du|^{2}}{2}\right) \bigg(\frac{1}{2d_F}\left( r\lambda^{\prime}(r)+ \lambda(r)\right)   |du |^2\\
			&+\left[\frac{1}{2d_F}\lambda(r)\left( \sigma-1-r\|\nabla \phi\|_{\infty}-2d_F\right)-r\lambda^{\prime}(r) \right]  \left|du\left(\frac{\partial}{\partial r}\right)\right|^{2}\bigg). 
		\end{split}
	\end{equation*}
Thus, we have 
\begin{equation*}
	\begin{split}
		0\geq&  \int_{ T_{R}(x)}F^\prime\left(\frac{|du|^{2}}{2}\right)\left[\frac{1}{2d_F} r\lambda^{\prime}(r)|du |^2-r\lambda^{\prime}(r)\left|du\left(\frac{\partial}{\partial r}\right)\right|^{2} \right] \\
		& +\int_{ B_{R/2}(x)}F^\prime\left(\frac{|du|^{2}}{2}\right) \bigg(\frac{1}{2d_F}|du |^2+\frac{1}{2d_F}\left( \sigma-1-r\|\nabla \phi\|_{\infty}-2d_F\right)  \left|du\left(\frac{\partial}{\partial r}\right)\right|^{2}\bigg) \\
			& +\int_{T_R(x)}F^\prime\left(\frac{|du|^{2}}{2}\right) \bigg(\frac{1}{2d_F}\lambda(r)|du |^2+\left[\frac{1}{2d_F}\lambda(r)\left( \sigma-1-r\|\nabla \phi\|_{\infty}-2d_F\right) \right]  \left|du\left(\frac{\partial}{\partial r}\right)\right|^{2}\bigg) \\
			\geq&  -\int_{ T_{R}(x)}F^\prime\left(\frac{|du|^{2}}{2}\right)(\frac{1}{2d_F}+1)C |du |^2 +\int_{ B_{R/2}(x)}F^\prime\left(\frac{|du|^{2}}{2}\right) \bigg(\frac{1}{2d_F}|du |^2\bigg) \\
			\geq&  -(2d_F+1)C\int_{ T_{R}(x)}F\left(\frac{|du|^{2}}{2}\right)  +\int_{ B_{R/2}(x)}F^\prime\left(\frac{|du|^{2}}{2}\right) \bigg(\frac{1}{2d_F}|du |^2\bigg). \\
			\end{split}
\end{equation*}
Let $ R\to \infty, $ we get 
\begin{equation*}
	\begin{split}
		du=0.
	\end{split}
\end{equation*}
\end{proof}

\begin{cor}
	If $ u: \left(M^{m}, f^{2} g_{0},e^{-\phi}\mathrm{d} v_g \right)\to (N,h)  $ be  $\phi$-$F$ harmonic map with finite energy, 
	$ K_r $ satisfies  one of the conditions (1)-(7) in Theorem \ref{thm1} .  Moreover,  	\begin{equation*}			\begin{split}
			&r\frac{\partial \log \eta }{\partial r}  (m-2d_F)+1+r \sum_{i=1}^{m-1} h_1(r) \geq \sigma>1,\\
			&\sigma-1-r\|\nabla \phi\|_{\infty}-2d_F\geq 0,\\
			&\frac{1}{2d_F}\left(\frac{\sigma-1}{r} \right)-h_2(r)-\|\nabla \phi\|_{\infty} \geq 0,
		\end{split}
	\end{equation*}
	where $ h_1(r)  $ and $ h_2(r) $ are lower bound and upper bound  of eigenfunction of   the tensor $ \operatorname{Hess}(r)$ in Lemma \eqref{56} and  Lemma \eqref{55} respectively,	then u is constant.

\end{cor}
\begin{proof}
	By \cite{wei2021dualities} or Lemma \eqref{56} and  Lemma \eqref{55}, we  can find the explicit formula for $ h_1(r) $ and  $ h_2(r). $
\end{proof}
	\section{Using the method in  Wang \cite{MR2959441} to deal with $\phi$-$F$-$V$-harmonic map from metric measure space }\label{sec6}
	In this section, we use the similar  conditions as  in \cite[Theorem 1.3]{MR2959441} to deal with.
		\begin{equation}
		\begin{split}
			&		\delta^\nabla\left( (F^{\prime}\left(\frac{|du|^{2}}{2}\right)du\right) -F^{\prime}\left(\frac{|du|^{2}}{2}\right)du(\nabla \phi-V)-\nabla H(u)=0.	
		\end{split}
	\end{equation}

	\begin{thm}\label{thm7}
		Let $\left(M^m, \eta^2 g_0, e^{-\phi} d v\right)$ be a metric measure space with a pole $ x_0 $. Let  $ r(x)=d_{g_0}(x,x_0) $.	 Assume that $(M, g_0)$ is a complete  noncompact simply connected Riemannian manifold with nonpositive sectional curvature $-a^2 \leq K_M(g_0) \leq 0$. Let $(N, h)$ be another Riemannian manifold, $b$ be a positive number and $\nabla_{\frac{\partial}{\partial r}} \phi \leq 0, \frac{\partial \log \eta}{\partial r}\geq 0$. Assume further that $\left(M, \eta^2 g_0, e^{-\phi} d v\right)$ satisfies one of the following conditions:
		
		(1) $\operatorname{Ric}^M(g_0) \leq-b^2$ and $b \geq 2 d_Fa,d_F \leq 1$;
		
		(2) $  \operatorname{Ric}_\phi^{\infty}(g_0)  \leq-b^2$ and $b \geq 2d_F a,d_F \leq 1$;
		
		(3) $\operatorname{Ric}_\phi^{\infty}(g_0)  \leq-b^2\left(\right.$ or $\left.\operatorname{Ric}^M(g_0)  \leq-b^2\right)$ and
		$$
		r \frac{\partial \phi}{\partial r} \leq 1+(b r) \operatorname{coth}(b r)+r\frac{\partial \log \eta}{\partial r}(m-2d_F)-2d_F ar \operatorname{coth}(a r) .
		$$
		
		Then 
		for $ R_2\geq R_1\geq r_0>0 $, we have 
		\begin{equation*}
			\begin{split}
				\frac{\int_{B_{R_1}\left(x_{0}\right) \backslash B_{r_{0}}\left(x_{0}\right)}  F(e(u)) e^{-\phi} d v }{R_1^{\delta_0}} \leq  	\frac{\int_{B_{R_2}\left(x_{0}\right) \backslash B_{r_{0}}\left(x_{0}\right)} F(e(u)) e^{-\phi} d v }{R_2^{\delta_0}} .
			\end{split}
		\end{equation*}
	Thus, any $\phi$-$F$-$V$-harmonic maps $ u : (M, \eta^2 g_0, e^{-\phi}\d v_g) \to (N, h) $ with finite $ F$-energy must 	be  a constant map.
	\end{thm}
	\begin{rem}
		The condition $ d_F \leq 1 $ can be satisfied if $ F(x)=(1+x)^{\alpha}, \alpha \leq 1. $ In 	\cite[Theorem C]{Liu2005}, Liu proved Liouville theorem for $ F $-harmonic map from $ (M,g) $ h slowly divergent 		$ F $-energy  if  $-a^2 \leq K_M \leq 0, \operatorname{Ric}^M(g) \leq-b^2.$
	\end{rem}
	\begin{proof}
		We  choose a orthornormal local coordinate $ \{e_i\} = \{e_s, \partial_r\} $ on $ B_R(x_0) $,
	let	$ r $ denote the geodesic distance function from center $ x_0 $.	Let $ S_{F} $ be 
		\begin{equation*}
			\begin{split}
				S_{F} =F\left(\frac{|du|^{2}}{2}\right) g-F^{\prime}\left(\frac{|du|^{2}}{2}\right)u^{*}h.
			\end{split}
		\end{equation*}
Let $ X=r \frac{\partial}{\partial r}, $	Then,
		\begin{equation}
			\int_{ B_{R}(x)}	\operatorname{div} \left( i_XS_{F}\right) e^{-\phi}dv_g =\int_{ B_{R}(x)}\langle S_{F},\nabla 
			X^\sharp \rangle_g e^{-\phi}dv_g+\int_{ B_{R}(x)}i_X \left( \operatorname{div} S_{F}\right) e^{-\phi}dv_g .
		\end{equation}	
		So we  have 		
		\begin{equation}\label{k1}
			\begin{split}
				\int_{ B_{R}(x)}	\operatorname{div} \left(e^{-\phi} i_XS_{F}\right) dv_g =&\int_{ B_{R}(x)}\langle S_{F},\nabla 
				X^\sharp \rangle_g e^{-\phi}dv_g-\int_{ B_{R}(x)}S(X,\nabla \phi) e^{-\phi}dv_g\\
				&+\int_{ B_{R}(x)}i_X \left( \operatorname{div} S_{F}\right) e^{-\phi}dv_g .
			\end{split}
		\end{equation}		
		The LHS of \eqref{k1} can be estimated by  
		\begin{equation}\label{986}
			\begin{split}
				\int_{ \partial B_{R}(x)} S_{F}(X, n) e^{-\phi}dv_g  \leq R \int_{\partial B_{R}\left(x_{0}\right)} F e^{-\phi} d S.
			\end{split}
		\end{equation}
		
		The RHS of \eqref{k1} is 
		\begin{equation*}
			\begin{split}
				&\langle S_{F},\nabla 
				X^\sharp \rangle_g -S(X,\nabla \phi)+i_X \left( \operatorname{div} S_{F,H,}\right) \\
				=&\langle S_{F},\nabla 
				X^\sharp \rangle_g -(F(\frac{|du |^2}{2})+H)\nabla_{X}\phi + F^\prime(\frac{|du |^2}{2}) u^*h(X,\nabla \phi)-\left \langle du(X),\tau_{F,H}(u) \right \rangle\\
				=&\langle S_{F},\nabla 
				X^\sharp \rangle_g -(F(\frac{|du |^2}{2})+H)\nabla_{X}\phi + F^\prime(\frac{|du |^2}{2}) h(du(X),du(V)).
			\end{split}
		\end{equation*}		
		So RHS of \eqref{k1} is 
		\begin{equation*}
			\begin{split}
				\int_{ B_{R}(x)} \langle S_{F},\nabla 
				X^\sharp \rangle_g e^{-\phi}dv_g-F(\frac{|du |^2}{2})\left( \nabla_{X}\phi\right)  e^{-\phi}dv_g.
			\end{split}
		\end{equation*}		
	As before, we have 	
		\begin{equation*}
			\begin{split}
					\left\langle S, \nabla X^{\sharp} \right\rangle=\left\langle S, r\frac{\partial \log \eta }{\partial r}g+\frac{1}{2}\eta ^2\mathcal{L}_X(g_0) \right\rangle.
			\end{split}
		\end{equation*}
		The integrand can be esimated as 
		\begin{equation*}
			\begin{split}
				&\eta ^2 e^{-\phi}	\left\langle S, \frac{1}{2}\mathcal{L}_X(g_0)\right\rangle_g-e^{-\phi}\bigg(F(\frac{|du |^2}{2})+H\bigg)\nabla_{X}\phi\\
				 					=& e^{-\phi}\left( F\left(\frac{|du|^{2}}{2}\right) +H\right)\left(1+r \Delta_\phi r\left(e_{i}, e_{i}\right)\right)-e^{-\phi}\F\left|du\left(\frac{\partial}{\partial r}\right)\right|^{2}\\
					 	-&e^{-\phi}\sum_{i,j=1}^{m-1}\F r \operatorname{Hess}_{g_0}(r)\left(e_{i}, e_{j}\right)\left\langle du\left(e_{i}\right), du\left(e_{j}\right)\right\rangle.
			\end{split}
		\end{equation*}
Thus, we have 		
		\begin{equation}\label{r1}
			\begin{split}
				&e^{-\phi}\langle S_{F},\nabla 
			X^\sharp \rangle_g-e^{-\phi}(F(\frac{|du |^2}{2}))\nabla_{X}\phi\\
					\geq&  e^{-\phi}\left( F\left(\frac{|du|^{2}}{2}\right) +H\right)\left(1+br\coth(br)-r\frac{\partial \phi}{\partial r}++r\frac{\partial \log \eta }{\partial r} (m-2d_F)\right)\\& -e^{-\phi}f\F\left|du\left(\frac{\partial}{\partial r}\right)\right|^{2}	-e^{-\phi}ar\coth(ar)\F\sum_{i=1}^{m-1}  \left\langle du\left(e_{i}\right), du\left(e_{i}\right)\right\rangle ,\\
				\geq &-e^{-\phi}\bigg[1-\frac{1}{2d_F}\bigg(1+br\coth(br)-r\frac{\partial \phi}{\partial r}+r\frac{\partial \log \eta }{\partial r} (m-2d_F)\bigg)\bigg]\F\left|du\left(\frac{\partial}{\partial r}\right)\right|^{2}\\& 	-e^{-\phi}\sum_{i,j=1}^{m-1} \bigg(ar\coth(ar)-\frac{1}{2d_F}\bigg[1+br\coth(br)-r\frac{\partial \phi}{\partial r}+r\frac{\partial \log \eta }{\partial r} (m-2d_F)\bigg]\bigg)\\
				&\times \F \left\langle du\left(e_{i}\right), du\left(e_{i}\right)\right\rangle ,\\
			\end{split}
		\end{equation}
		where $ d_F \leq 1, r>0 $.

		Let $ r_0 $ such that $  r_0\coth(r_0)>0 $.	There exists a positive constant $ \delta_0, r_0$ such that (c.f. \cite[(3.11)]{MR2959441}), wehen $ r(x) \geq r_0 $		
		\begin{equation*}
			\begin{split}
				-\bigg(ar\coth(ar)-\frac{1}{2d_F}\bigg[1+br\coth(br)-r\frac{\partial \phi}{\partial r}+r\frac{\partial \log \eta }{\partial r} (m-2d_F)\bigg]\bigg)\geq \frac{\delta_{0}}{2},\\
				-\bigg[1-\frac{1}{2d_F}\bigg(1+br\coth(br)-r\frac{\partial \phi}{\partial r}+r\frac{\partial \log \eta }{\partial r} (m-2d_F)\bigg)\bigg]\geq  \frac{\delta_{0}}{2}.
			\end{split}
		\end{equation*}
 Hence, we get 	
		\begin{equation}\label{9888}
			\begin{split}
				&\int_{B_{R}\left(x_{0}\right) \backslash B_{r_{0}}\left(x_{0}\right)}\left(\left(S_{F}(u), \nabla X\right)-e^{-\phi}F(\frac{|du |^2}{2}) e(u)\nabla_{X} \phi\right) d v\\
				\geq & \left( \delta_{0}-R\|V\|_{\infty}\right)  \int_{B_{R}\left(x_{0}\right) \backslash B_{r_{0}}\left(x_{0}\right)}  F^{\prime}(\frac{|du |^2}{2}) e(u) e^{-\phi} d v,
			\end{split}
		\end{equation}
where  positive number $\delta_{0}>0$ which depends on $r_{0}$. Combining  \eqref{986} and \eqref{9888}, for $R \geq r_{0}$ we get 
		\begin{equation}\label{r9}
			\begin{split}
				R \int_{\partial B_{R}\left(x_{0}\right)} F e^{-\phi} d S \geq& \left( \delta_{0}-R\|V\|_{\infty}\right)  \int_{B_R(x_0)} F^{\prime}(\frac{|du |^2}{2}) e(u)e^{-\phi}\d v_g\\
				\geq &2l_F\left( \delta_{0}-R\|V\|_{\infty}\right)  \int_{B_R(x_0)} F(\frac{|du |^2}{2})e^{-\phi} \d v_g,
			\end{split}
		\end{equation}
		which implies
		\[
		\frac{d}{d R} \int_{B_{R}\left(x_{0}\right) \backslash B_{r_{0}}\left(x_{0}\right)}  F(\frac{|du |^2}{2}) e^{-\phi} d v \geq \frac{\delta_{0} }{R} \int_{B_{R}\left(x_{0}\right) \backslash B_{r_{0}}\left(x_{0}\right)}  e(u) e^{-\phi} d v ,  R \geq R_{0}.
		\]
		So for $ R_2\geq R_1\geq R_0 $,  we have 
		\begin{equation*}
			\begin{split}
				\frac{\int_{B_{R_1}\left(x_{0}\right) \backslash B_{r_{0}}\left(x_{0}\right)}  e(u) e^{-\phi} d v }{R_1^{\delta_0}} \leq  	\frac{\int_{B_{R_2}\left(x_{0}\right) \backslash B_{r_{0}}\left(x_{0}\right)}  e(u) e^{-\phi} d v }{R_2^{\delta_0}}. 
			\end{split}
		\end{equation*}		
		By \eqref{r9}, as in \cite{MR2959441},  we can prove that $ u $ is constant.

	Secondly, we prove the Theorem if the condtion (2) holds. By \cite[Corollary A.1]{MR2959441}
		\begin{equation}
			\begin{split}
			&e^{-\phi}\langle S_{F},\nabla 
			X^\sharp \rangle_g-e^{-\phi}(F(\frac{|du |^2}{2}))\nabla_{X}\phi\\
				&	\geq  e^{-\phi}\left( F\left(\frac{|du|^{2}}{2}\right) \right)\left(1+br\coth(br)+r\frac{\partial \log \eta }{\partial r} (m-2d_F)\right)-e^{-\phi}f\F\left|du\left(\frac{\partial}{\partial r}\right)\right|^{2}\\& 	-e^{-\phi}\sum_{i,j=1}^{m-1}f\F  ar\coth(ar)\left\langle du\left(e_{i}\right), du\left(e_{i}\right)\right\rangle ,\\
				\geq &-e^{-\phi}\bigg[1-\frac{1}{2d_F}\bigg(1+br\coth(br)+r\frac{\partial \log \eta }{\partial r} (m-2d_F)\bigg)\bigg]\F\left|du\left(\frac{\partial}{\partial r}\right)\right|^{2}\\& 	-e^{-\phi}\sum_{i,j=1}^{m-1}\F  \bigg(ar\coth(ar)-\frac{1}{2d_F}\bigg[1+br\coth(br)+r\frac{\partial \log \eta }{\partial r} (m-2d_F)\bigg]\bigg)\left\langle du\left(e_{i}\right), du\left(e_{i}\right)\right\rangle .\\
			\end{split}
		\end{equation}
Let $ r_0 $ such that $  r_0\coth(r_0)>0 $. 	However, the condtion (2) implies that  there exists a positive constan. $ \delta_0, r_0$ such that (c.f. \cite[(3.11)]{MR2959441}), when $ r(x) \geq r_0 $
		\begin{equation*}
			\begin{split}
				-\bigg(ar\coth(ar)-\frac{1}{2d_F}\bigg[1+br\coth(br)+r\frac{\partial \log \eta }{\partial r} (m-2d_F)\bigg]\bigg)\geq \frac{\delta_{0}}{2},\\
				-\bigg[1-\frac{1}{2d_F}\bigg(1+br\coth(br)+r\frac{\partial \log \eta }{\partial r} (m-2d_F)\bigg)\bigg]\geq  \frac{\delta_{0}}{2},
			\end{split}
		\end{equation*}

	 Lastly, we prove the Theorem if the condtion (3) holds. By \eqref{r1}, as in \cite{MR2959441}, we can prove that $ u $ is independent of $ r $.  By \cite[Corollary A.2]{MR2959441}
		$$
		\begin{aligned}
			\int_M F(\frac{|du |^2}{2}) e^{-\phi} d v & \geq  \int_{r_0}^R \int_{S^{n-1}}F(\frac{|du |^2}{2}) A(r, \theta) d r d \theta_{n-1} \\
			& \geq C e^{b R} \int_{S^{n-1}}F(\frac{|du |^2}{2})d \theta_{n-1}.
		\end{aligned}
		$$
		So if $\int_{S^{n-1}}F(\frac{|du |^2}{2}) d \theta_{n-1} \neq 0$, we have $\int_M F(\frac{|du |^2}{2}) e^{-\phi} d v=\infty$. This contradiction  yields $e(u) \equiv 0$, so $u$ is a constant map.
		
	\end{proof}
	\begin{thm} Under the same conditions as in Theorem \ref{thm7}, if the energy of $ u $ satisfies
		$ 	\int_{B_R(x_0)} F(\frac{|du |^2}{2})e^{-\phi}\d v_g = o(R^{\delta_0}) $ as $ R \to \infty $, then $ u $ is a constant map. In particular,
		if the energy of $ u $ is finite or moderate divergent, then any $\phi$-$F$-harmonic maps $u:\left(M, g, e^{-\phi} d v\right) \rightarrow(N, h)$ with finite $F$-energy must be constant map.
	\end{thm}
	
	\begin{proof}
		As in Theorem \ref{thm7},		
		\[
		R \int_{\partial B_{R}\left(x_{0}\right)}F(\frac{|du |^2}{2}) e^{-\phi} d S \geq 2d_F\left( \delta_{0}-R\|V\|_{\infty}\right)  C,
		\]
		and		
		\begin{equation*}
			\begin{split}
				\lim _{R \rightarrow \infty} \int_{B_{R}\left(x_{0}\right)} \frac{1}{\psi(r(x))}F(\frac{|du |^2}{2})e^{-\phi} &=\int_{0}^{\infty}\left(\frac{1}{\psi(R)} \int_{\partial B_{R}\left(x_{0}\right)} F(\frac{|du |^2}{2})\right)e^{-\phi} \d R \\
				& \geq  2d_F\left( \delta_{0}-R\|V\|_{\infty}\right)  C\int_{0}^{\infty} \frac{\d R}{R \psi(R)} \\
				& \geq 2d_F\left( \delta_{0}-R\|V\|_{\infty}\right)C \int_{R_{0}}^{\infty} \frac{\d R}{R \psi(R)}=\infty.
			\end{split}
		\end{equation*}		
		This  contradicts to the assumption that the energy of $ u $ is finite or moderate divergent.

	\end{proof}

		\section{$ \phi $-$ V $ harmonic map with potential } \label{sec7}
	In this section, we deal with   $\phi$-$V$ harmonic map with  potential from metric measure space .  Let $ e(u)=\frac{1}{2}|du |^2. $

	\begin{lem}
Let $u$ be  $\phi$-$V$ harmonic map with  potential from $(M,g)$ to $(N,h)$. Let $S=e(u)g-u^{*}h+H(u)g,$  then 		
		\begin{equation*}
			\begin{split}
				\operatorname{div} (S)(X)	&=\left\langle  du(V),du(X)\right\rangle. 
			\end{split}
		\end{equation*}		
	\end{lem}
	\begin{lem}
		\begin{equation*}
			\begin{split}
				\left\langle S, \nabla X^{b}\right\rangle &\geq  \left( e(u)+H\right)\\
				&\times
				\left(1+r \sum_{i=1}^{m-1} h_1(r)-4r  h_2(r)+(m-4) r\frac{\partial \log \eta }{\partial r} \right).
			\end{split}
		\end{equation*}
	\end{lem}
	\begin{proof}Similar to  the proof of formula \eqref{e1}, we let $ X=\nabla(\frac{1}{2}r^2),  g=\eta ^2 g_0$
		$$\nabla  X^{b}=1/2\mathcal{L}_X(g)= 1/2\mathcal{L}_X(g)=r\eta \frac{\partial \eta }{\partial r}g_0+\frac{1}{2}\eta ^2\mathcal{L}_X(g_0)=r\frac{\partial \log \eta }{\partial r}g+\frac{1}{2}\eta ^2\mathcal{L}_X(g_0).$$
		Let ${\left\{e_{i}\right\}_{i=1}^{m}}$ be an orthonormal frame with respect to ${g_{0}}$ and ${e_{m}=\frac{\partial}{\partial r}}$. We may assume that  ${\operatorname{Hess}_{g_{0}}\left(r^{2}\right)}$ is a diagonal matrix with respect to ${\left\{e_{i}\right\}}$. Note that ${\left\{\hat{e_{i}}=\eta^{-1} e_{i}\right\}}$ is an orthonormal frame with respect to ${g}$. Therefore	as in the proof of formula \eqref{e1},  
		\begin{equation}\label{}
			\begin{split}
				& \eta ^2	\left\langle S, \frac{1}{2}\mathcal{L}_X(g_0)\right\rangle_g \\
				&	= \left( e(u) +H\right)\left(1+r \sum_{i=1}^{m-1} \operatorname{Hess}_{g_0}(r)\left(e_{i}, e_{i}\right)\right)-f\left|du\left(\frac{\partial}{\partial r}\right)\right|^{2}\\& 	-\sum_{i,j=1}^{m-1}r \operatorname{Hess}_{g_0}(r)\left(e_{i}, e_{j}\right)\left\langle du\left(e_{i}\right), du\left(e_{j}\right)\right\rangle h,\\
				&\geq  \left( e(u) +H\right)\left(1+r \sum_{i=1}^{m-1} h_1(r)- r  h_2(r)\right) -f\left|du\left(\frac{\partial}{\partial r}\right)\right|^{2}(1-rh_2(r))\\
				&\geq  \left( e(u) +H\right)\left(1+r \sum_{i=1}^{m-1} h_1(r)-2r  h_2(r)\right) ,
			\end{split}
		\end{equation}
		and
		\begin{equation*}
			\begin{split}
				r\frac{\partial \log \eta }{\partial r}	\left\langle S, g\right\rangle&= 	r\frac{\partial \log \eta }{\partial r}\left\langle e(u) g-fu^{*}h+H(u)g,g\right\rangle \\ 
				&\geq  r\frac{\partial \log \eta }{\partial r} \left( (m-2) e(u)+mH\right) \\
				&\geq (m-2) r\frac{\partial \log \eta }{\partial r} \left[   e(u)+H \right]  .
			\end{split}
		\end{equation*}
	\end{proof}
	
	\begin{thm}\label{thm8.1}
		Let $ (M^m, g=f^2g_0,e^{-\phi }\d V_g) $ be a complete Riemannian manifold with a pole $  x_0 $,
		$ V \in \Gamma(TM). $ Let $ (N^n, h) $ be a Riemannian manifold. 
		Let $  u : M \to N $ be a $\phi$-$V$-harmonic map with potential . We assume that
		$$ |\nabla \phi | \leq
		\frac{C}{2r}, H>0,r\frac{\partial \log f }{\partial r}\geq 0, $$
		where $  C < \sigma $ is a constant. Assume that the radial curvature $K(r)$	of $M$ satisfies one of  the conditions (1)-(7) in Theorem \ref{thm1}
		
		\begin{align}
			\frac{ \int_{ B_{R_1}(x)}\big(  e(u)+H(u)\big)e^{-\phi}}{R_1^{\sigma-C}}\leq \frac{ \int_{ B_{R_2}(x)}\big(  e(u)+H(u)\big)e^{-\phi}}{R_2^{\sigma-C}}.
		\end{align}
		
		where
		\begin{equation*}
			\begin{split}
				\sigma=\begin{cases} 1+(m-1)A_1-2 A, & \text { if } K(r) \text { satisfies (1), } \\
					1+(m-1) \frac{1+\sqrt{1+4 A_{1}}}{2}- (1+\sqrt{1+4 A}), & \text { if } K(r) \text { satisfies (2), } \\
					1+(m-1)\left(\left|B-\frac{1}{2}\right|+\frac{1}{2}\right)- 2d_{F}\left(1+\sqrt{1+4 B_{1}\left(1-B_{1}\right)}\right), & \text { if } K(r) \text { satisfies (3), } \\
					1+(m-1) \frac{1+\sqrt{1-4 B}}{2}- \left(1+\sqrt{1+4 B_{1}}\right), & \text { if } K(r) \text { satisfies (4), } \\
					m-2 \frac{\alpha}{\beta} , \quad& \text { if } K(r) \text { satisfies (5), } \\
					m-2, & \text { if } K(r) \text { satisfies (6), } \\
					m-(m-1) \frac{B}{2 \epsilon}-2  \mathrm{e}^{\frac{A}{2 \epsilon}} , & \text { if } K(r) \text { satisfies (7). }
				\end{cases}
			\end{split}
		\end{equation*}
	\end{thm}
	\begin{rem}
		Using this monotonicity formula ,we can get Liouville theorem as in Theorem \ref{cnm} if the engergy is finite in some sense or moderate divergent. The proof is standard, we omit it.
	\end{rem}

	\begin{proof}

		%
		%
		%
		
		Next we use  the conservation law as in the proof of Theorem \ref{thm1}. 
		\begin{equation}\label{e22}
			\int_{ B_{R}(x)}	\operatorname{div} \left( i_X(S)\right)e^{-\phi}\d v_g =\int_{ B_{R}(x)}\langle S,\nabla 
			X^\sharp \rangle_g e^{-\phi}\d v_g+\int_{ B_{R}(x)}i_X \left( \operatorname{div}(S)\right)e^{-\phi}\d v_g .
		\end{equation}

		so we  have 
		
		\begin{equation}\label{k112}
			\begin{split}
				\int_{ B_{R}(x)}	\operatorname{div} \left(e^{-\phi} i_X(S)\right) dv_g =\int_{ B_{R}(x)}\langle S,\nabla 
				X^\sharp \rangle_g e^{-\phi}dv_g-\int_{ B_{R}(x)}S(X,\nabla \phi) e^{-\phi}dv_g\\
				+\int_{ B_{R}(x)}i_X \left( \operatorname{div} (S)\right) e^{-\phi}dv_g .
			\end{split}
		\end{equation}
		
		The LHS of \eqref{k112} is 
		\begin{equation}\label{}
			\begin{split}
				\int_{ \partial B_{R}(x)} S(X, n) e^{-\phi}dv_g  \leq R \int_{\partial B_{R}\left(x_{0}\right)} (e(u)+H) e^{-\phi} d S.
			\end{split}
		\end{equation}
		
		The RHS of \eqref{k112} is 
		\begin{equation}\label{et3}
			\begin{split}
				&-(S)(X,\nabla \phi)+i_X \left( \operatorname{div} (S)\right) \\
				=&-(F(\frac{|du |^2}{2})+H)\nabla_{X}\phi +  u^*h(X,\nabla \phi)-\left \langle du(X),\tau_{F,H}(u) \right \rangle\\
				=&-(F(\frac{|du |^2}{2})+H)\nabla_{X}\phi +\langle du(X), du(V)-du(\nabla \phi)\rangle, 
			\end{split}
		\end{equation}
		
		and 
		\begin{equation*}
			\begin{split}
				&\int_{ B_{R}(x)}\langle S_u,\nabla 
				X^\sharp \rangle_g e^{-\phi}\d v_g\\
				\geq&  \int_{ B_{R}(x)} \left( e(u) +H\right)\left(1+r \sum_{i=1}^{m-1} h_1(r)-2r  h_2(r)+(m-2) r\frac{\partial \log \eta }{\partial r} \right)\\
				\geq& \sigma  	\int_{ B_{R}(x)} \left( e(u)+H\right)  e^{-\phi}\d v_g.
			\end{split}
		\end{equation*}
		However, 	
		\begin{equation}\label{et3}
			\begin{split}
				&\int_{B_R(x)}\bigg[-(S(X,\nabla \phi)+i_X \left( \operatorname{div} (S)\right)\bigg]  e^{-\phi}\d v_g\\
				\geq &\int_{B_R(x)}\langle \d u(X), \d u(V)\rangle e^{-\phi}\d v_g \\
				\geq &-R\|V\|_{\infty}\int_{ B_{R}(x)}\left( e(u)+H(u)\right) e^{-\phi}\d v_g. 
			\end{split}
		\end{equation}
		These imply that

		\begin{align}
			R \frac{d}{d R}  \int_{ B_{R}(x)}\left(  e(u)+H(u)\right)e^{-\phi}\d v_g  \geq (\sigma- C)\int_{ B_{R}(x)}\left( e(u)+H(u)\right) e^{-\phi}\d v_g .
		\end{align}
		This can be rewritten as
		\begin{align}
			\frac{d}{d R} \frac{ \int_{ B_{R}(x)}\left(  e(u)+H(u)\right)e^{-\phi}\d v_g }{R^{\sigma-C}} \geq 0.
		\end{align}
		This is 
		\begin{equation}\label{}
			\begin{split}
				\frac{ \int_{ B_{R_1}(x)}\left(   e(u)+H\right)e^{-\phi}\d v_g}{R_1^{\sigma-C}}
				\leq \frac{ \int_{ B_{R_2}(x)}\left(   e(u) +H\right)e^{-\phi}\d v_g}{R_2^{\sigma-C}}.
			\end{split}
		\end{equation}
		
	\end{proof}

	\section{  Using the method in Zhao \cite{zhao2019monotonicity} to deal with $ \phi $-$ p $ harmonic map from metric measure space    }\label{sec8}
	In this section, we use the method in \cite{zhao2019monotonicity} to study  $ p $ harmonic map.
	\begin{thm}\label{thm9.1}
		Let $ (M^m, g=f^2g_0,e^{-\phi}dv_g) $ be a complete metric measure space  with a pole $  x_0 $.
		Let $ (N^n, h) $ be a Riemannian manifold. Assume that the radial
		curvature $ K_r $ of $ M $ with respec to $ g_0 $ satisfies one of the conditions (1),(2),(3) of Lemma \ref{56}.
		Let $  u : M \to N $ be the solution of 
		\begin{equation*}
			\begin{split}
				\delta^\nabla \left( |du|^{p-2}du\right) )-|du|^{p-2}du(\nabla \phi)+\nabla H=0.
			\end{split}
		\end{equation*}
		We assume that
		$$ |\nabla \phi| \leq
		\frac{C}{2r},r\frac{\partial \log f }{\partial r}\geq 0, p\leq m. $$
		where   $  C < \Lambda $ is a constant . Then for $ R_1 \leq R_2, $ we  have 
	\begin{align}
		\frac{ \int_{ B_{R_1}(x)} e_p(u)e^{-\phi}dv_g}{R_1^{\Lambda-pC}}\leq \frac{ \int_{ B_{R_2}(x)}  e_p(u)e^{-\phi}dv_g}{R_2^{\Lambda-pC}}.
	\end{align}
	
		where $e_p(u)=\frac{|du|^p}{p}$ and 
		\begin{equation*}
			\begin{split}
				\Lambda=\begin{cases}
					m-\frac{p \alpha}{\beta},&   \text{if} \quad K_r \quad satisfies\quad (1) \\
					m-1-(m-1)(1-\dfrac{B}{2\epsilon})-pe^{\frac{A}{2\epsilon}},&   \text{if} \quad K_r \quad  satisfies \quad (2) \\
					\frac{m-1}{2}\left(1+\sqrt{1-4 b^{2}}\right)-\sqrt{1+4 a^{2}} ,&  \text{if}\quad  K_r \quad satisfies \quad (3) 
				\end{cases}
			\end{split}
		\end{equation*}
	\end{thm}	
		\begin{rem}
		Using this monotonicity formula ,we can get Liouville theorem as in Theorem \ref{cnm} if the engergy is finite in some sense or moderate divergent. The proof is standard, we omit it.
	\end{rem}

	\begin{proof}For $ \phi $-$ p $ harmonic map, we define the tensor 
		\begin{equation*}
			\begin{split}
				S=\frac{|du|^p}{p} g-|du|^{p-2} u^{*}h+H(u)g.\\
			\end{split}
		\end{equation*}
		Take $ F=\frac{x^{\frac{p}{2}}\sqrt{2}^p}{p} $ in Lemma \ref{lem41}, we have 
		\begin{equation*}
			\begin{split}
				\left(\operatorname{div} S\right)(X)
				&=\langle\tau_{p},du(X)\rangle=\langle |du|^{p-2}du(\nabla \phi),du(X)\rangle,\\
			\end{split}
		\end{equation*}
		where  $ \tau_{p}= \delta^\nabla \left( |du|^{p-2}du\right) ).$
		As before, 
		\begin{equation*}
			\begin{split}
				\nabla  X^{b}=\frac{1}{2}\mathcal{L}_X(g)=&r\frac{\partial \log f}{\partial r}g_0+\frac{1}{2}f^2\mathcal{L}_X(g_0)\\
						\end{split}
		\end{equation*}
		
		Let ${\left\{e_{i}\right\}_{i=1}^{m}}$ be an orthonormal frame with respect to ${g_{0}}$ and ${e_{m}=\frac{\partial}{\partial r}}$. We may assume that  ${\operatorname{Hess}_{g_{0}}\left(r^{2}\right)}$ is a diagonal matrix with respect to ${\left\{e_{i}\right\}}$. Note that ${\left\{\hat{e_{i}}=f^{-1} e_{i}\right\}}$ is an orthonormal frame with respect to ${g}$. Inspired by the method  in \cite[Theorem 3.2]{zhao2019monotonicity}, we also consider the three cases.
		
	If condition (1) holds, then we have 
		\begin{equation}\label{}
			\begin{split}
			f^2\left\langle S, \frac{1}{2}\mathcal{L}_X(g_0)\right\rangle &\geq  \left( e_{p}(u)\right) (1+(m-1) \beta r \operatorname{coth}(\beta r))\\ 
				&-|du|^{p-2}\langle du(\frac{\partial}{\partial r}),du(\frac{\partial}{\partial r})\rangle-|du|^{p-2}\alpha r \operatorname{coth}(\alpha r)\langle du(e_i),du(e_i)\rangle \\
				&=\left[\frac{1}{p}+\frac{m-1}{p} \beta r \operatorname{coth}(\beta r)-1\right]\left|d u\left(\frac{\partial}{\partial r}\right)\right|^{2} |du |^{p-2}\\
				& +\left[\frac{1}{p}+\frac{m-1}{p} \beta r \operatorname{coth}(\beta r)-\alpha r \operatorname{coth}(\alpha r)\right]\left\langle d u\left(e_{i}\right), d u\left(e_{i}\right)\right\rangle |du |^{p-2}\\
				&=[(m-1) \beta r \operatorname{coth}(\beta r)+1-p] \cdot \frac{1}{p}\left|d u\left(\frac{\partial}{\partial r}\right)\right|^{2}|du |^{p-2} \\
				&+\left[1+\beta r \operatorname{coth}(\beta r)\left(m-1-\frac{p\alpha r \operatorname{coth}(\alpha r)}{\beta r \operatorname{coth}(\beta r)}\right)\right] \cdot \frac{1}{p}\left\langle d u\left(e_{i}\right), d u\left(e_{i}\right)\right\rangle |du |^{p-2}\\
				&\geq \left(m-\frac{p \alpha}{\beta}\right) e_{p}(u)  ,
			\end{split}
		\end{equation}
		where we have used the face that $ \beta r \operatorname{coth}(\beta r)\geq 1, \coth(x)   $ is  nonincreasing function, and $ \alpha \leq \dfrac{m-1}{p}\beta .$
		
		\begin{equation}\label{cfg}
			\begin{split}
				\left\langle S, \frac{1}{2}\mathcal{L}_X(g)\right\rangle_{g}\geq& r\frac{\partial \log f}{\partial r}\left\langle S_{u},g\right\rangle+  \left(m-\frac{p \alpha}{\beta}\right) e_{p}(u)\\
				=&(m-p)r\frac{\partial \log f}{\partial r}e_{p,g}(u)+  \left(m-\frac{p \alpha}{\beta}\right) e_{p}(u)\\
				\geq & \left(m-\frac{p \alpha}{\beta}+(m-p)r\frac{\partial \log f}{\partial r}\right)e_{p}(u), 
			\end{split}
		\end{equation}
		%
		If condition (2) holds, then we have 
		\begin{equation}\label{}
			\begin{split}
					f^2\left\langle S, \frac{1}{2}\mathcal{L}_X(g_0)\right\rangle &\geq  e_{p}(u)(1+(m-1) (1-\frac{B}{2\epsilon}))\\ 
				&-|du|^{p-2}\langle du(\frac{\partial}{\partial r}),du(\frac{\partial}{\partial r})\rangle-|du|^{p-2}e^{\dfrac{A}{2\epsilon}}\langle du(e_i),du(e_i)\rangle \\
				&=|du|^{p-2}\left( (m-1)(1-\dfrac{B}{2\epsilon})+1-p\right) \dfrac{1}{p}\langle du(\frac{\partial}{\partial r}),du(\frac{\partial}{\partial r})\rangle\\
				&+|du|^{p-2}\left( (m-1)(1-\dfrac{B}{2\epsilon})+1-pe^{\frac{A}{2\epsilon}}\right) \dfrac{1}{p}\langle du(\frac{\partial}{\partial r}),du(\frac{\partial}{\partial r})\rangle\\
				&\geq \left(m-1-(m-1)(1-\dfrac{B}{2\epsilon})-pe^{\frac{A}{2\epsilon}}\right) e_p(u),
			\end{split}
		\end{equation}
		where, we have used the assumption $m-(m-1)\dfrac{B}{2\epsilon}-pe^{\frac{A}{2\epsilon}}>0,e^{\frac{A}{2\epsilon}}\geq 1.$
		
		As in \eqref{cfg}, we have 
		\begin{equation}\label{}
			\begin{split}
				&\left\langle S, \frac{1}{2}\mathcal{L}_X(g)\right\rangle_{g}\\
				\geq
				& \left(m-1-(m-1)(1-\dfrac{B}{2\epsilon})-pe^{\frac{A}{2\epsilon}}+(m-p)r\frac{\partial \log f}{\partial r})\right) e_{p}(u), 
			\end{split}
		\end{equation}
		%
		If condition (3) holds, then we have 	
		\begin{equation}\label{}
			\begin{split}
				f^2\left\langle S, \frac{1}{2}\mathcal{L}_X(g_0)\right\rangle \geq & e_p(u)(1+(m-1) \frac{1+\sqrt{1 - 4b^2}}{2})\\ 
				&-|du|^{p-2}\langle du(\frac{\partial}{\partial r}),du(\frac{\partial}{\partial r})\rangle-|du|^{p-2}\frac{1+\sqrt{1 + 4a^2}}{2}\langle du(e_i),du(e_i)\rangle \\
				=& {\left[\frac{m-1}{4}\left(1+\sqrt{1-4 b^{2}}\right)-\frac{1}{2}\right]\left|d u\left(\frac{\partial}{\partial r}\right)\right|^{2} } |du|^{p-2}\\
				&+\left[\frac{1}{2}+\frac{m-1}{4}\left(1+\sqrt{1-4 b^{2}}\right)-\frac{1+\sqrt{1+4 a^{2}}}{2}\right]\left\langle d u\left(e_{i}\right), d u\left(e_{i}\right)\right\rangle |du|^{p-2}\\
				\geq&\left(\frac{m-1}{2}\left(1+\sqrt{1-4 b^{2}}\right)-\sqrt{1+4 a^{2}}\right) e_p(u) .
			\end{split}
		\end{equation}
		As in \eqref{cfg}, we have 
	\begin{equation}\label{}
		\begin{split}
			&\left\langle S, \frac{1}{2}\mathcal{L}_X(g)\right\rangle_{g}\\
			\geq
			& \left(\frac{m-1}{2}\left(1+\sqrt{1-4 b^{2}}\right)-\sqrt{1+4 a^{2}}+(m-p)r\frac{\partial \log f}{\partial r})\right) e_{p}(u), 
		\end{split}
	\end{equation}
		%
		
		
		Next we emply the conservation law as in the proof of Theorem \ref{thm1}. 
		Thus, by \eqref{e222}, we  have 		
		\begin{equation}\label{k1122}
			\begin{split}
				\int_{ B_{R}(x)}	\operatorname{div} \left(e^{-\phi} i_X(S)\right) dv_g =\int_{ B_{R}(x)}\langle S,\nabla 
				X^\sharp \rangle_g e^{-\phi}dv_g-\int_{ B_{R}(x)}S(X,\nabla \phi) e^{-\phi}dv_g\\
				+\int_{ B_{R}(x)}i_X \left( \operatorname{div} (S)\right) e^{-\phi}dv_g .
			\end{split}
		\end{equation}
		
		The LHS of \eqref{k1122} is 
		\begin{equation}\label{}
			\begin{split}
				\int_{ \partial B_{R}(x)} (S)(X, n) e^{-\phi}dv_g  \leq R \int_{\partial B_{R}\left(x_{0}\right)} (e_p(u)+H) e^{-\phi} d S\\
				= R \frac{d}{d R}  \int_{ B_{R}(x)} e_p(u)+H(u)e^{-\phi}dv_g.
			\end{split}
		\end{equation}
		
		Notice that 
		\begin{equation}
			\begin{split}
				&-S(X,\nabla \phi)+i_X \left( \operatorname{div} (S)\right) \\
				= &-(e_p(u)+H)\nabla_{X}\phi +  u^*h(X,\nabla \phi)-\left \langle du(X),\tau_{p,H}(u) \right \rangle\\
				=&-(e_p(u)+H)\nabla_{X}\phi+\langle |du|^{(p-2)}du(\nabla \phi) , du(X)\rangle  . 
			\end{split}
		\end{equation}

		The RHS of \eqref{k1122} is  bigger than 	
		\begin{equation}\label{}
			\begin{split}
				\int_{B_{R}(x)} \operatorname{ div}( S_u)(X)=& \int_{B_{R}(x)}\left(\langle |du|^{(p-2)}du(\nabla \phi) , du(X)\rangle\right)e^{-\phi}dv_g\\
				\leq&  p|\nabla \phi|_{\infty}R \int_{B_{R}(x)} e_p(u)+H(u)e^{-\phi}dv_g\\
				\leq & pC \int_{B_{R}(x)} e_p(u)+H(u)e^{-\phi}dv_g.
			\end{split}
		\end{equation}
		Thus, we have 
		\begin{align}
			R \frac{d}{d R}  \int_{ B_{R}(x)}\left(  e_p(u)+H(u)\right) \geq (\Lambda- pC)\int_{ B_{R}(x)}\left(  e_p(u)+H(u)\right) .
		\end{align}
		This can be rewritten as
		\begin{align}
			\frac{d}{d R} \frac{ \int_{ B_{R}(x)}e_{p}(u)e^{-\phi}dv_g}{R^{\Lambda-pC}} \geq 0.
		\end{align}
		Thus,  for $0 \leq  R_1 \leq R_2, $ we get 
		\begin{align}
			\frac{ \int_{ B_{R_1}(x)} e_p(u)e^{-\phi}dv_g}{R_1^{\Lambda-pC}}\leq \frac{ \int_{ B_{R_2}(x)}  e_p(u)e^{-\phi}dv_g}{R_2^{\Lambda-pC}}.
		\end{align}

	\end{proof}
	
	\bibliographystyle{acm}
	\bibliography{mybib2022}

\end{document}